\documentclass[12pt,reqno]{amsart}

\allowdisplaybreaks

\usepackage{amsfonts}
\usepackage{amsmath}
\usepackage{mathtools}
\usepackage{amssymb}
\usepackage{hyperref,thmtools}
\usepackage{cleveref}
\usepackage{amsthm}
\usepackage{verbatim}
\usepackage{ulem}
\usepackage{cancel}
\usepackage{tikz, tikz-cd}
\usepackage{ytableau}
\usepackage{quiver}

\theoremstyle{plain}
\newtheorem{theorem}{Theorem}[section]
\newtheorem{proposition}[theorem]{Proposition}
\newtheorem{lemma}[theorem]{Lemma}
\newtheorem{corollary}[theorem]{Corollary}

\theoremstyle{definition}

\theoremstyle{remark}
\newtheorem{remark}[theorem]{Remark}
\theoremstyle{example}
\newtheorem{example}[theorem]{Example}

\numberwithin{equation}{section}

\usepackage[]{cite}

\newcommand\CC{\mathbb{C}}

\newcommand\ZZ{\mathbb{Z}}

\newcommand\one{\mathbf{1}}

\newcommand\Inv{\mathrm{Inv}}

\newcommand\leg{\mathrm{leg}}

\newcommand\GT{\mathrm{GT}}

\newcommand{\Leg}{\mathrm{Leg}}


\title{The monomial expansion formula for \\ Hall-Littlewood $P$-polynomials}
\author{Aritra Bhattacharya}
\email{matharitra@gmail.com}
\date{\today}

\begin{document}
	
	\maketitle

\begin{abstract}
	We give a Hecke algebra derivation of Macdonald’s expansion formula for Hall-Littlewood polynomials
	in terms of semistandard Young tableaux.  This is accomplished by first obtaining a Hecke algebra lift of the
	expansion coefficients and then proving a generalization of Klostermann’s recursions.
\end{abstract}

\section{Introduction}


\subsection{}The Hall-Littlewood polynomials are a simultaneous generalization of the Schur polynomials, which are characters of the irreducible rational representations of $GL_n(\CC)$ and the monomial symmetric polynomials, the naive basis of the space of symmetric polynomials, obtained from a monomial by symmetrizing over the symmetric group $S_n$. In fact, when the space of symmetric polynomials is seen as the spherical Hecke algebra through Satake isomorphism, the Hall-Littlewood polynomials become the naive basis of this space, where one obtains these polynomials by Hecke symmetrizing a monomial. In this article, we obtain another proof of the monomial expansion formula for Hall-Littlewood polynomials due to Macdonald, by use of an affine Hecke algebra lift. 


\subsection{}The Hall-Littlewood $P$-polynomials is a family $(P_\lambda(t) : \lambda \in (\ZZ^n)_+)$ of polynomials that form a graded $\ZZ[t^{\pm1}]$-basis of the space of symmetric polynomials in $n$ variables $\ZZ[t^{\pm 1}][X_1^{\pm1},\ldots,X_n^{\pm 1}]^{S_n}$, which is the $t$-analog of the character ring of rational representations of $GL_n(\CC)$. Two important specializations are the Schur polynomials $s_\lambda$, and the monomial symmetric polynomials $m_\lambda$.  



\[\begin{tikzcd}
	{s_\lambda} && {P_\lambda(t)} && {m_\lambda}
	\arrow["{t=0}"{description}, from=1-3, to=1-1]
	\arrow["{t=1}"{description}, from=1-3, to=1-5]
\end{tikzcd}\]

The transition matrices between these three bases are well-studied. The transition matrix coefficients between the Schur polynomials and the monomial symmetric polynomials are given by the Kostka numbers, which counts the number of semistandard Young tableaux of a given shape and content. The transition matrix coefficients between the Schur polynomials and the Hall-Littlewood polynomials are given by the Kostka-Foulkes polynomials, which counts the semistandard tableaux with a statistic called charge. Both the Kostka numbers and the Kostka-Foulkes polynomials appear in ubiquity in combinatorial, representation theoretic and geometric contexts. For a survey of Hall-Littlewood and Kostka-Foulkes polynomials see \cite{MacMainBook}, \cite{NelsenRam}, \cite{DLTsurvery94}.

\subsection{}This paper focuses on the transition matrix between the Hall-Littlewood polynomials and the monomial symmetric polynomials. In \cite{MacMainBook}, Macdonald obtained the monomial expansion formula for $P_\lambda(t)$ as a sum over semistandard Young tableaux of shape $\lambda$, and the weight associated to each tableau $T$ is a polynomial $\psi_T \in \ZZ[t]$. Macdonald's formula reads \begin{equation}\label{Macformula}
P_\lambda(t) = \sum_{T \in 
	B(\lambda)} \psi_T x^T \,, \qquad \hbox{ for } \lambda \in (\ZZ_{\geq 0}^n)_+\,,
\end{equation} 
where $B(\lambda)$ denotes the set of semistandard Young tableaux of shape $\lambda$ with fillings from $[n]$. 
The definition of semistandard Young tableaux and Macdonald's $\psi_T$ is recalled in \S\ref{tableauxdef}, \S\ref{sec:Macformula}.

In fact, the Hall-Littlewood polynomials $P_\lambda(t)$ may be defined more generally as an element of $\ZZ[t][P]^W$, where $P$ is the weight lattice and $W$ the Weyl group of a given root system or a root datum of a reductive linear algebraic group, and $\lambda$ varies over $P_+$, the set of dominant integral weights. The $P_\lambda(t)$ discussed in previous paragraphs are the type $GL_n$ Hall-Littlewood polynomials.

In \cite{GaussentLittelmann_oneskeletongalleries}, Gaussent and Littelmann obtained a monomial expansion formula for the Hall-Littlewood polynomials as a sum over positively folded one-skeleton galleries in the standard apartment of the affine building. Their formula works for all types. Klostermann \cite{Klostermann} showed that for type $GL_n$ the Gaussent-Littelmann formula and Macdonald's formula agrees term-by-term via a bijection between the positively folded one-skeleton galleries and the semistandard Young tableaux.

The Gaussent-Littelmann formula has as its predecessor the formula by Schwer in \cite{Schwer}. Schwer's formula expresses the monomial coefficients in $P_\lambda(t)$ as a sum of positively folded alcove galleries in the standard apartment of the affine building. An exposition of Schwer's formula can be found in \cite{Ram_alcovewalksHeckealgebras}, where positively folded galleries are replaced with positively folded alcove walks. In \cite{RamYip}, the alcove walk formula were generalized to give a monomial expansion formula for the Macdonald polynomials, which are generalizations of Hall-Littlewood polynomials. Schwer's approach to calculating the transition matrix coefficients between $(P_\lambda(t))_\lambda$ and $(m_\lambda)_\lambda$ is through an affine Hecke algebra lift, which we now describe. 

\subsection{}Throughout the paper our exposition will be in type $GL_n$ for some fixed $n \in \ZZ_{>0}$, although some (but not all) of the statements make sense in general type.

Let $H$ be the affine Hecke algebra (defined in \S\ref{ahadef}). The set $\{ X^{\mu} T_w \,|\, \mu \in \ZZ^n, w \in S_n \}$ is a $\ZZ[t^{\pm 1}]$-basis for $H$. Define the finite Hecke symmetrizer $$ \one_0 = \sum_{w \in S_n} T_w \,.$$ Then the Satake isomorphism theorem says that $$ \one_0 H \one_0 \cong \ZZ[t^{\pm 1}][X_1^{\pm1},\ldots,X_n^{\pm 1}]^{S_n}\,, $$ and the Hall-Littlewood polynomials $P_\lambda(t)$ are defined as the image of the element $\dfrac{1}{W_\lambda(t)} \one_0 X^\lambda \one_0$, where $W_\lambda(t)$ is a polynomial in $t$, such that the coefficient of $X^\lambda$ in $P_\lambda(t)$ is $1$. 

Schwer first expresses the element $\one_0 X^\lambda$ of $H$ in the basis $\{ X^{\mu} T_w \,|\, \mu \in \ZZ^n, w \in S_n \}$ of $H$. The formula for $P_\lambda(t)$ is then obtained by right multiplication with $\one_0$ and using $T_w \one_0 = t^{\ell(w)}\one_0$ for $w \in S_n$.

Gaussent and Littelmann in \cite{GaussentLittelmann_oneskeletongalleries} obtained another formula for the monomial expansions of $P_\lambda(t)$, using a `geometric compression' to Schwer's formula. Klostermann \cite{Klostermann} showed that this formula agrees with Macdonald's formula \eqref{Macformula}. Klostermann's proof uses the fact that for a tableau $T$ with columns $C_{r},\ldots,C_1$ (from left-to-right), $$ \psi_T = \prod_{j = 1}^{r-1} \psi_{C_{j+1}\otimes C_j} \,, \qquad \hbox{( see \eqref{eq:psiTproduct} for a proof )}\,,$$ where $C_{j+1} \otimes C_j$ denotes the two column tableau with columns $C_{j+1}$ and $C_j$. For two column tableaux, $\psi$ satisfies the recursions in \autoref{prop:klostermannrecs}. These recursions look exactly like the recursions for $R$-polynomials in Kazhdan-Lusztig theory, see for example Theorem 5.1.1 in \cite{BjornerBrenti} and compare with \eqref{eq:tildepsirec2cols}.


\subsection{}In this article we follow Schwer and Ram's idea to compute the monomial expansion of $P_\lambda(t)$, we first write $\one_0 X^\lambda$ in the basis $\{ X^{\mu} T_w \,|\, \mu \in \ZZ^n, w \in S_n \}$ of $H$. 

Let $\varpi_k$ be the $k$th fundamental weight (see \eqref{varpikdef} for the notation). For $\mu \in \ZZ^n$ let $S_{n,\mu} = \{w \in S_n | w\mu = \mu \}$.

Let $T \in B(\varpi_{\ell_r})\otimes \ldots \otimes B(\varpi_{\ell_1})$ with $1 \leq \ell_1 \leq 
\ldots \leq \ell_r \leq n$. We define an element $\Psi_T$ in the finite Hecke algebra $H_n$ by a recursive definition. 

For a column $C \in B(\varpi_\ell)$, define \begin{equation}
	\Psi_C = t^{\ell(u_C)}(T_{u_C^{-1}})^{-1} \one_{\varpi_\ell},
\end{equation}
where $\one_{\varpi_{\ell}}$ is the parabolic Hecke symmetrizer $\sum\limits_{w \in S_{n,\varpi_{\ell}}} T_w$.

Suppose we have obtained $\Psi_{C_{k}\otimes\ldots\otimes C_1}$ where $C_i \in B(\varpi_{\ell_i})$. Then we write $$ \Psi_{C_{k}\otimes\ldots\otimes C_1} = \sum_{C \in B(\varpi_{\ell_{k+1}})} T_{u_C} h_{C,C_{k},\ldots,C_1}\,, \qquad \hbox{ with } \qquad h_{C,C_k,\ldots,C_1} \in H_{n,\varpi_{\ell_{k+1}}} \,,$$ where for $\mu \in \ZZ^n$, $H_{n,\mu}$ is the subalgebra of $H$ generated by $\{ T_w \,|\, w \in S_{n,\mu} \}$, and $u_C$s are the minimal length coset representatives of $S_n/S_{n,\varpi_{\ell_{k+1}}}$. Then define \begin{equation}
	\Psi_{C_{k+1}\otimes\ldots\otimes C_1} = t^{\ell(u_{C_{k+1}})}(T_{u_{C_{k+1}}^{-1}})^{-1}h_{C_{k+1},\ldots,C_1}\,.
\end{equation}

Let $H_n$ be the finite Hecke algebra inside the affine Hecke algebra $H$. For a column strict filling $T$ of partition shape the above paragraph defines an element $\Psi_T \in H_n$.


A tableau is called highest weight if all entries in the $i$th row is equal to $i$. We show in \autoref{PsiT0ifnotSSYT} and \autoref{Psiinitialcondition} that \begin{theorem}
	Let $\lambda = \varpi_{\ell_r}+ \ldots + \varpi_{\ell_1}$ with $1 \leq \ell_1 \leq 
	\ldots \leq \ell_r \leq n$. Let $T \in B(\varpi_{\ell_r})\otimes \ldots \otimes B(\varpi_{\ell_1})$ . Then \begin{enumerate}
		\item[(a)] If $T$ is not semistandard then $\Psi_T = 0$.
		\item[(b)] If $T$ is highest weight then $\Psi_T = \one_\lambda$. 
	\end{enumerate} 
\end{theorem} 


Our first main result is \autoref{10Xlambda}, which says that $$ \one_0 X^\lambda = \sum_{T \in B(\lambda)} X^T \Psi_T \,,$$ this is an affine Hecke algebra lift of Macdonald's formula \eqref{Macformula}. 

%

We then describe the elements $\Psi_T$ by a recursion in \autoref{Psirecrcols}. Using this, we show that if $$\widetilde{\psi}_T = \dfrac{1}{W_\lambda(t)}\Psi_T\one_0\,, \qquad \hbox{ for } T \in B(\lambda) \,,$$ then $\widetilde{\psi}_T \in \ZZ[t^{\pm1}]$ and $$\widetilde{\psi}_T = \psi_T\,, \qquad \hbox{ for } T\in B(\lambda)\,,$$ i.e, our $\widetilde{\psi}$ agrees with Macdonald's $\psi$ (\autoref{tildepsi=psi}). To show this, we first show that if $T$ has columns $C_r,\ldots,C_1$ (from left to right) then $$ \widetilde{\psi}_T = \prod_{j = 1}^{r-1}\widetilde{\psi}_{C_{j+1}\otimes C_{j}}\,, \qquad (\hbox{\eqref{eq:tildepsiTproduct}})  \,.$$ Then we find a recursion for $\widetilde{\psi}_{F\otimes E}$ in \eqref{eq:tildepsirec2cols} which agrees with the recursion for $\psi_{F\otimes E}$ from \autoref{prop:klostermannrecs}. Then our $\widetilde{\psi}_T$ and Macdonald's $\psi_T$ satisfies the same recurrence and intital conditions, this proves their equality.

\subsection{}Finally, we remark that just like the methods used in Schwer's formula were generalized to obtain the alcove walk formula for Macdonald polynomials, we hope that our methods may in future generalize to give a compressed alcove walk formula for Macdonald polynomials.

\subsection*{Acknowledgements}

We are very grateful to Arun Ram for suggesting this project and providing his initial computations, valuable insights and comments. We thank S. Viswanath for helpful discussions.

\section{The Symmetric Group}

In this section we review some well-known facts about the symmetric group seen from a Coxeter theoretic point of view. We will use results from this section to deduce identities in the affine Hecke algebra.

Fix $n \in \ZZ_{>0}$.

\subsection{Symmetric Group}
For two integers $1 \leq a \leq b \leq n$, let $[a,b]$ denote the interval $\{a,a+1,\ldots,b\}$, and $[n] = [1,n]$. Let $S_{[a,b]}$ denote the symmetric group on $[a,b]$, i.e, $S_{[a,b]}$ is the set of all bijections $[a,b] \to [a,b]$, and $S_n = S_{[1,n]}$. Then $S_{[a,b]} \subseteq S_{[1,n]} = S_n$. 

Let $s_i$ denote the simple transposition $(i,i+1)$ in $S_n$. Then the set $s_1,\ldots,s_{n-1}$ form a presentation of the group $S_n$ with the following relations.
\begin{align}
	&s_i^2 = 1\,, \qquad \hbox{ for } i \in [n-1]\,,
	\\
	&s_is_{i+1}s_i = s_{i+1}s_is_{i+1}\,, \qquad \hbox{ for } i \in [n-1]\,, 
	\label{Snbraid1}
	\\
	&s_is_j = s_js_i\,, \qquad \hbox{ for } i,j \in [n-1], j \notin \{i\pm 1\}\,.
	\label{Snbraid2}
\end{align}

\subsection{Action on $\ZZ^n$}

The symmetric group $S_n$ acts on $\ZZ^n$ by permuting co-ordinates. 

 Let $\varepsilon_k = (0,\ldots,0,1,0,\ldots,0) \in \ZZ^n$ be the vector with $1$ in the $k$-th position and $0$ everywhere else. Equip $\ZZ^n$ with the standard bilinear form $(\varepsilon_i, \varepsilon_j) = \delta_{i,j}$ for $i,j \in [n]$\,. This form is $S_n$-invariant: $(w\alpha,w\beta) = (\alpha,\beta)$ for $w \in S_n$, $\alpha,\beta \in \ZZ^n$. 

Let $$ \alpha_i = \varepsilon_i - \varepsilon_{i+1}\,, \qquad \hbox{ for } i \in [n-1] \,.$$

The action of $s_i$ on $\ZZ^n$ can be described as  \begin{equation}\label{siactZn}
	s_i \beta = \beta - (\beta,\alpha_i) \alpha_i \,, \qquad \hbox{ for } i \in [n-1], \beta \in \ZZ^n \,.
\end{equation}
In particular, if $(\beta,\alpha_i) = 0$ then $s_i\beta = \beta$. 

\subsection{The length function} Every element $w \in S_n$ can be expressed as a word in $s_1,\ldots,s_{n-1}$. The \textit{length} of $w$, $\ell(w)$ is the length of the smallest word for $w$. A \textit{reduced expression} for $w$ is an expression for $w$ with length $\ell(w)$.

Let $$ R_+ = \{\varepsilon_i - \varepsilon_j \,|\, 1 \leq i < j \leq n\} \qquad \text{ and } \qquad R_{-} = - R_+ \,.$$

For $w \in S_n$, the set of \textit{inversions} of $w$ is $$\Inv(w) = \{ \gamma \in R_+ \,|\, w \gamma \in R_{-} \}\,.$$

Let $w = s_{i_1}\ldots s_{i_\ell}$ be a reduced expression for $w$. Then \cite[Chap. VI, \S 1.6, Corollary 2]{Bourbaki} says \begin{equation}
	\Inv(w) = \{  s_{i_\ell}\ldots s_{i_{j+1}} \alpha_{i_j} \,|\, j \in [\ell]\}\,.
	\label{eq:Invset}
\end{equation}  
Consequently, \begin{equation}\label{len=inv}
	\ell(w) = \# \Inv(w)\,. 
\end{equation}

\subsection{Parabolic subgroups and quotients}\label{subsec:parabolicsubgrquotient} 

Let \begin{equation}\label{Znplusdef}
	(\ZZ^n)_+ = \{ \lambda \in \ZZ^n \,|\, \lambda_1 \geq \ldots \geq \lambda_n \} \,, \, \hbox{ and } \, (\ZZ_{\geq 0}^n)_+ = \{\lambda \in \ZZ_{\geq 0}^n \,|\, \lambda_1 \geq \ldots \geq \lambda_n\}\,.
\end{equation}

For $\lambda \in \ZZ_{\geq 0}^n$ denote by $S_{n,\lambda}$ the stabilizer subgroup $\{w \in S_n \,|\, w\lambda = \lambda \}$ and let $S_n^\lambda$ be the set of minimal length coset representatives of $S_n/S_{n,\lambda}$. If $\lambda = (n^{m_n},\ldots,1^{m_1}) = (\underbrace{n,\ldots,n}_{m_n},\ldots,\underbrace{1,\ldots,1}_{m_1})$, then $S_{n,\lambda} = S_{[1,m_n]}\times \ldots \times S_{[n-m_1+1,n]}$. For $1 \leq a \leq b \leq n$, the subgroup $S_{[a,b]}$ of $S_n$ is isomorphic to $S_{b-a+1}$ and is generated by $s_{a},\ldots,s_{b-1}$ , and $S_{n,\lambda}$ is generated by $\{ s_i \,|\, s_i \lambda = \lambda \}$ . 


The following proposition will be crucial in our analysis.

\begin{proposition}{\sc\cite[Proposition 2.4.4]{BjornerBrenti}}
	\label{Snfactorization}
	Let $\lambda \in (\ZZ_{\geq 0}^n)_+$. Then every element $w \in S_n$ has a unique factorization as $w = uv$ with $u\in S_n^\lambda, v \in S_{n,\lambda}$. Moreover, $\ell(w) = \ell(u) + \ell(v)$.
	
\end{proposition}

	\begin{lemma}\label{lambetpos}
		Let $\lambda \in (\ZZ_{\geq 0}^n)_+$ and $u \in S_n^\lambda$ and $\beta \in \Inv(u)$. Then $(\lambda , \beta)> 0 $.
	\end{lemma}
	
	\begin{proof}
		$(\lambda , \beta) \geq 0$ since $\lambda \in (\ZZ_{\geq 0}^n)_+$ and $\beta \in R_+$. We will show that if $(\lambda, \beta) = 0$ then there is an element shorter than $u$ in $uS_{n,\lambda}$.
		
		Suppose $u = s_{i_1}\ldots s_{i_\ell}$ is a reduced expression for $u$. Since $\beta \in \Inv(u)$, by \eqref{eq:Invset}, $\beta = s_{i_\ell}\ldots s_{i_{j+1}}\alpha_{i_j}$ for some $j \in [\ell]$.
		
		If $(\lambda,\beta) = 0$, then by the $S_n$-invariance of $(-,-)$, $(s_{i_{j+1}}\ldots s_{i_\ell}\lambda , \alpha_{i_j}) = 0$. So by \eqref{siactZn}, $s_{i_j}(s_{i_{j+1}}\ldots s_{i_\ell}\lambda) = s_{i_{j+1}}\ldots s_{i_\ell}\lambda$, which gives $s_{i_\ell}\ldots s_{i_j}\ldots s_{i_\ell} \lambda = \lambda$, i.e, $s_{i_\ell}\ldots s_{i_j}\ldots s_{i_\ell} \in S_{n,\lambda}$. Then $s_{i_j}\ldots s_{i_\ell}S_{n,\lambda} = s_{i_{j+1}}\ldots s_{i_\ell}S_{n,\lambda}$. So $uS_{n,\lambda} = s_{i_1}\ldots s_{i_{j-1}}s_{i_{j+1}}\ldots s_{i_\ell} S_{n,\lambda}$. 
	\end{proof}
	
	Let \begin{equation}\label{varpikdef}
		\varpi_k = \varepsilon_1+\ldots+\varepsilon_k\,,\qquad \hbox{ for } k \in [n]\,.
	\end{equation} 

	\begin{corollary}\label{omegak1}
		Let $u \in S_n^{\varpi_k}$ with a reduced expression $u = s_{i_1}\ldots s_{i_\ell}$. \\Then $ (s_{i_{j+1}}\ldots s_{i_\ell} \varpi_k , \alpha_{i_j} ) = 1$ for every $j \in [\ell]$.
	\end{corollary}
	
	\begin{proof}
		For $j \in [\ell]$, $ ( s_{i_{j+1}}\ldots s_{i_\ell} \varpi_k , \alpha_{i_j} ) = ( \varpi_k , s_{i_\ell}\ldots s_{i_{j+1}} \alpha_{i_j} ) > 0 \,.$
		Since $ (\varpi_k,\alpha) \in \{0,1\}$ for $\alpha \in R_+ \,$, hence $( s_{i_{j+1}}\ldots s_{i_\ell} \varpi_k , \alpha_{i_j} ) = 1\,$.
	\end{proof}

\subsection{Columns}\label{columndef}

A \textit{column} of length $\ell$ is a sequence $I = (i_1,\ldots,i_\ell)$ such that $1 \leq i_1 <\ldots <i_\ell\leq n$. Denote the set of columns of length $\ell$ by $B(\varpi_\ell)$. We can draw a column $I \in B(\varpi_{\ell})$  as a filling of the Young diagram of the partition $\varpi_{\ell} = (1,\ldots,1,0,\ldots,0)$. 

For example, if $n \geq 4$ then $$\ytableausetup{nosmalltableaux,centertableaux} (1,2,4) = \ytableaushort{1,2,4}  \in B(\varpi_3) \,.$$

A column of length $\ell$ can be identified with a subset of size $\ell$ from $[n]$. For $I \in B(\varpi_\ell)$, let $I^c$ be the column of length $n-\ell$ whose entries are the elements of the complementary set $[n] - I$, written in increasing order.

If $I = (i_1,\ldots,i_\ell)\in B(\varpi_\ell)$ and $I^c = (i^c_{\ell+1},\ldots,i^c_n)$ then let \begin{equation}
	u_I = \begin{pmatrix}
	1 & \ldots & \ell & \ell+1 & \ldots &n\\ i_1 & \ldots & i_\ell & i^c_{\ell+1} & \ldots & i^c_n
\end{pmatrix}.
\end{equation}

\begin{proposition}{\sc\cite[\S 2.4]{BjornerBrenti}}\label{Snparabolicquotients}
	For $\ell \in [n]$, $S_n^{\varpi_\ell} = \{u_I: I \in B(\varpi_\ell)\}$.
\end{proposition}


Define a partial order $\leq$ on $B(\varpi_{\ell})$ as follows. 

\noindent Let $E = (e_1,\ldots,e_\ell), F = (f_1,\ldots,f_\ell) \in B(\varpi_{\ell})$. Define \begin{equation}
	\label{eq:colposetdef}
	E \leq F \qquad\hbox{ if }\qquad e_i \leq f_i \hbox{ for all } i \in [\ell]\,.
\end{equation}
\autoref{fig:Bvarpi3} shows the Hasse diagram of the poset $B(\varpi_3)$ when $n = 5$.

Next, we define an action of the symmetric group $S_n$ on $B(\varpi_{\ell})$.

\noindent We can identify a column $C = (c_1,\ldots,c_\ell)$ with the element $\vec{C} = \varepsilon_{c_1}+\ldots+ \varepsilon_{c_\ell} \in \ZZ^n$. Then define $s_jC$ to be the column obtained from  $s_j\vec{C}$. 

\noindent Equivalently,
\begin{quote}
	if either both $j,j+1 \in C$ or both $j,j+1 \notin C$ then $s_jC = C$,
	\\
	if $j \in C, j+1 \notin C$, then $s_jC$ is obtained by replacing $j$ with $j+1$ in $C$,	
	\\
	if $j \notin C, j+1 \in C$, then $s_jC$ is obtained by replacing $j+1$ with $j$ in $C$.
\end{quote}


\noindent This defines an action of the symmetric group $S_n$ on $B(\varpi_{\ell})$.

\begin{example}
	\ytableausetup{centertableaux,nosmalltableaux}
	\begin{align*}
		\hbox{If } C = \ytableaushort{1,2,5}\,, \quad \hbox{ then } \quad s_1C = \ytableaushort{1,2,5}\,,\quad s_2C = \ytableaushort{1,3,5}\,,\quad s_3C = \ytableaushort{1,2,5}\,, \quad s_4C = \ytableaushort{1,2,4}\,.
	\end{align*}
\end{example}

The following lemma is a consequence of the definition of the partial order $\leq$ on columns and the action of the symmetric group on columns, as defined above.
\begin{lemma}\label{lemma:sjEcomparison}
	Let $\ell \in [n]$ and $E \in B(\varpi_{\ell})$. Then 
		\item[(1)] $s_jE < E$ if and only if $j \notin E$, $j+1 \in E$ if and only if $(\vec{E},\alpha_j)=-1$,
		\item[(2)] $s_jE > E$ if and only if $j \in E$, $j+1 \notin E$ if and only if $(\vec{E},\alpha_j)=+1$,
		\item[(3)] $s_jE = E$ if and only if either both $j, j+1 \in E$ or both $j,j+1 \notin E$ if and only if $(\vec{E},\alpha_j)=0$.
	
\end{lemma}

The following lemma is fundamental to all our proofs.
\begin{lemma}\label{lemma:sjuE}
	Let $E \in B(\varpi_{\ell})$ for some $\ell \in [n]$. Then
	\begin{align*}
		\ell(s_ju_E) = \begin{cases}
			\ell(u_E) - 1\,,& \hbox{ if } j \notin E, j+1 \in E,
			\\
			\ell(u_E) + 1\,,& \hbox{ otherwise}.
		\end{cases}
	\end{align*}
	Moreover, if $s_jE \neq E$ then $s_ju_E = u_{s_jE}$, and if $s_jE = E$ then $s_ju_E = u_Es_k$ where $k = u_E^{-1}(j)$.
\end{lemma}
\begin{proof}
		\item[(1)] 
		Suppose $j \notin E, j+1 \in E$. Since multiplying by $s_j$ on the left interchanges the preimage of $j$ and $j+1$, $s_ju_E$ has one less inversion than $u_E$. 
		\begin{align*}
			&u_E = \begin{pmatrix}
				1& \cdots\cdots\cdots & \ell & \ell+1 & \cdots\cdots\cdots & n \\ e_1& \cdots j+1 \cdots & e_\ell & e^c_{\ell+1} & \cdots j \cdots & e^c_{n} 
			\end{pmatrix}, 
			\\
			&s_ju_E = \begin{pmatrix}
				1& \cdots\cdots\cdots & \ell & \ell+1 & \cdots\cdots\cdots & n \\ e_1& \cdots j \cdots & e_\ell & e^c_{\ell+1} & \cdots j+1 \cdots & e^c_{n} 
			\end{pmatrix}.
		\end{align*}
		Since $s_jE$ is obtained by replacing $j+1$ by $j$ in $E$, $s_ju_E = u_{s_jE}$.
		
		\item[(2)] If $s_jE = E$, then either both $j,j+1 \in E$ or both $j,j+1 \notin E$. Since multiplying by $s_j$ on the left interchanges the preimage of $j$ and $j+1$, $s_ju_E$ has one more inversion than $u_E$.
		{\begin{align*}
			&u_E = \begin{pmatrix}
				1& \cdots\cdots\cdots & \ell & \ell+1 & \cdots & n \\ e_1& \cdots j, j+1 \cdots & e_\ell & e^c_{\ell+1} & \cdots  & e^c_{n} 
			\end{pmatrix}, 
			\\
			&s_ju_E = \begin{pmatrix}
				1& \cdots\cdots\cdots & \ell & \ell+1 & \cdots & n \\ e_1& \cdots j+1, j \cdots & e_\ell & e^c_{\ell+1} & \cdots  & e^c_{n} 
			\end{pmatrix},
\\
			&u_E = \begin{pmatrix}
				1& \cdots & \ell & \ell+1 & \cdots\cdots\cdots & n \\ e_1& \cdots & e_\ell & e^c_{\ell+1} & \cdots j, j+1 \cdots & e^c_{n} 
			\end{pmatrix}, 
			\\
			&s_ju_E = \begin{pmatrix}
				1& \cdots & \ell & \ell+1 & \cdots\cdots\cdots & n \\ e_1& \cdots & e_\ell & e^c_{\ell+1} & \cdots j+1, j \cdots & e^c_{n} 
			\end{pmatrix}.
		\end{align*}}
		
		Since $u_E^{-1}s_ju_E$ is a transposition and $u_E^{-1}s_ju_E(u_E^{-1}(j)) = u_E^{-1}(j+1) = u_E^{-1}(j)+1$, hence $s_ju_E = u_Es_k$ where $k = u_E^{-1}(j)$.
		
		\item[(3)] Suppose $j \in E, j+1 \notin E$. Since multiplying by $s_j$ on the left interchanges the preimage of $j$ and $j+1$, $s_ju_E$ has one more inversion than $u_E$.
		\begin{align*}
			&u_E = \begin{pmatrix}
				1& \cdots\cdots\cdots & \ell & \ell+1 & \cdots\cdots\cdots & n \\ e_1& \cdots j \cdots & e_\ell & e^c_{\ell+1} & \cdots j+1 \cdots & e^c_{n} 
			\end{pmatrix}, 
			\\
			&s_ju_E = \begin{pmatrix}
				1& \cdots\cdots\cdots & \ell & \ell+1 & \cdots\cdots\cdots & n \\ e_1& \cdots j+1 \cdots & e_\ell & e^c_{\ell+1} & \cdots j \cdots & e^c_{n} 
			\end{pmatrix}.
		\end{align*}
		Since $s_jE$ is obtained by replacing $j$ by $j+1$ in $E$, $s_ju_E = u_{s_jE}$.

\end{proof}

\subsection{The Bruhat order}

The Bruhat order $\leq$ is a partial order on $S_n$. For two elements $u,v \in S_n$, $u \leq v$ if any (or some) reduced expression for $v$ contains a subexpression for $u$. For more on Bruhat order see \cite[Chapter 2]{BjornerBrenti}. In particular, if $w$ and $i$ is such that $\ell(s_iw)>\ell(w)$ then $s_iw>w$.

For a subset $I$ of $[n]$ of size $k$, let $I_\leq$ be the ordered tuple $(i_1,\ldots,i_k)$ of increasing arrangement of $I$, i.e, $i_1 = \min I$ and  $i_j = \min(I \setminus \{i_1,\ldots,i_{j-1}\} )$ for $j \in [2,k]$. Then for a subset $I$ of size $k$ of $[n]$, $I_{\leq} \in B(\varpi_k)$.

For $w \in S_n$, let $w[k] = \{w_1,\ldots,w_k\}$. The following theorem is known as the tableau criterion for Bruhat order.
\begin{theorem}{\sc\cite{BjornerBrenti_TableauCriterion} \cite[\S 5.9]{HumphreysCoxeterGroups}}\label{bruhattableaucondition}
	Let $u,v \in S_n$. Then $u \leq v$ in Bruhat order if and only if $u[k]_{\leq} \leq v[k]_{\leq}$ in the partial order on $B(\varpi_k)$ for all $k \in [n]$.
\end{theorem}

Suppose $L_1,L_2 \in B(\varpi_k)$ for some $k \in [n]$. Then by \cite[Proposition 2.4.8]{BjornerBrenti}, $L_1 \leq L_2$ in the partial order on columns from \eqref{eq:colposetdef} if and only if $u_{L_1} \leq u_{L_2}$ in the Bruhat order.

\begin{figure}[h]
	\[\begin{tikzcd}[row sep= 0.6cm]\ytableausetup{smalltableaux}
	& \ytableaushort{3,4,5} \\
	& {\ytableaushort{2,4,5}} \\
	{\ytableaushort{2,3,5}} && {\ytableaushort{1,4,5}} \\
	{\ytableaushort{2,3,4}} && {\ytableaushort{1,3,5}} \\
	{\ytableaushort{1,3,4}} && {\ytableaushort{1,2,5}} \\
	& {\ytableaushort{1,2,4}} \\
	& {\ytableaushort{1,2,3}}
	\arrow["{s_2}"', from=2-2, to=1-2]
	\arrow["{s_3}", from=3-1, to=2-2]
	\arrow["{s_1}"', from=3-3, to=2-2]
	\arrow["{s_4}", from=4-1, to=3-1]
	\arrow["{s_1}", from=4-3, to=3-1]
	\arrow["{s_3}"', from=4-3, to=3-3]
	\arrow["{s_1}", from=5-1, to=4-1]
	\arrow["{s_4}"', from=5-1, to=4-3]
	\arrow["{s_2}"', from=5-3, to=4-3]
	\arrow["{s_2}", from=6-2, to=5-1]
	\arrow["{s_4}"', from=6-2, to=5-3]
	\arrow["{s_3}", from=7-2, to=6-2]
\end{tikzcd}\]	
	\caption{The Hasse diagram of the poset $B(\varpi_3)$ when $n= 5$.}
	\label{fig:Bvarpi3}
\end{figure}
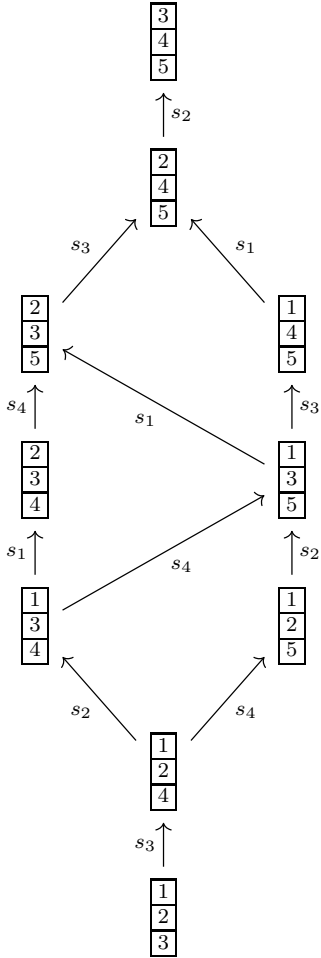


\newpage

\section{The affine Hecke algebra}

This section presents the definition of the type $GL_n$ affine Hecke algebra through generators and relations. We review the finite Hecke symmetrizer $\one_0$, which is a $t$-analog of the element $\sum\limits_{w \in S_n}w$ in the group algebra of the symmetric group $S_n$. We present with complete proofs various commutation relations between the $T_i$s and the $X_j$s. Finally we recall the Satake isomorphism theorem which says that the spherical Hecke algebra is isomorphic to the space of symmetric polynomials which in turn is the ring of rational representations of $GL_n(\CC)$.

\subsection{The affine Hecke algebra}\label{ahadef} The \textit{affine Hecke algebra} (AHA) is the $\ZZ[t^{\pm 1}]$ algebra $H$ with generators $T_1,\ldots,T_{n-1}$ and $X_1^{\pm 1},\ldots X_n^{\pm 1}$ subject to the following relations 	\begin{align}
		&T_i^2 = (t-1)T_i +t\,, \qquad \hbox{ for } i \in [n-1],
	\label{quadTi}
	\\
	&T_iT_{i+1}T_i = T_{i+1}T_i T_{i+1}\,, \qquad \hbox{ for } i \in [n-2],
	\label{braid1}
	\\
	& T_iT_j = T_jT_i\,, \qquad \hbox{ for } i,j \in [n-1], |i-j|>1,
	\label{braid2}
\end{align}
and \begin{align}\label{XiXjcommutes}
	X_i X_j = X_j X_i\,, \qquad \hbox{ if } i,j \in [n], 
\end{align}
and
\begin{align}
	&T_i X_i T_i = tX_{i+1}\,, \qquad \qquad \hbox{ if } i \in [n-1],
	\\
	&T_i X_j = X_j T_i\,, \qquad \hbox{ if } i \in [n-1],j \in [n],\hbox{ and } j \notin \{i,i+1\}.
	\label{XaffHeckerelsF}
\end{align}

\vspace{1em}
We will use the following equivalent form of \eqref{quadTi} \begin{equation}\label{tT_iinv}
	tT_i^{-1} = T_i+1-t\,.
\end{equation} 

For $w \in S_n$, let $w = s_{i_1}\ldots s_{i_\ell}$ be a reduced expression. Then define $T_w = T_{i_1}\ldots T_{i_m}$. Since any two reduced expressions for $w$ are related by the braid relations \eqref{Snbraid1}, \eqref{Snbraid2} and $T_1,\ldots,T_{n-1}$ satisfies the braid relations \eqref{braid1},\eqref{braid2}, $T_w$ is a well defined element of $H$.  

Then \begin{align}
	T_iT_w = \begin{cases}
		T_{s_iw}\,, &\hbox{ if } s_iw>w,
		\\
		T_iT_iT_{s_iw}\,, &\hbox{ if } s_iw<w,
	\end{cases}
	= \begin{cases}
		T_{s_iw}\,, &\hbox{ if } s_iw>w,
		\\
		(t-1)T_{w} + tT_{s_iw}\,, &\hbox{ if } s_iw<w.
	\end{cases}
	\label{T_iT_w}
\end{align}

In particular, \begin{equation}\label{TuveqTuTv}
	T_{uv} = T_uT_v\,, \qquad \hbox{ if } u,v\in S_n \, ,\, \ell(uv) = \ell(u) + \ell(v).
\end{equation}
%
%

\subsection{The finite Hecke symmetrizer}

Let \begin{equation}\label{10def}
	\one_0 = \sum_{w\in S_n}T_w\,.
\end{equation}

Using \eqref{T_iT_w}, \begin{align*}
	T_i\one_0 &= \sum_{w \in S_n}T_i T_w = \sum_{w: s_iw>w} (T_iT_w + T_iT_{s_iw}) =  \sum_{w: s_iw>w} (T_{s_iw} + (t-1)T_{s_iw}+tT_w) = t\one_0 .
\end{align*}
Therefore,
\begin{align}\label{Tw10}
	T_w \one_0 = t^{\ell(w)}\one_0\,, \qquad \hbox{ for } w\in S_n\,, 
\end{align}
and \begin{align}
	\one_0^2 = \sum_{w \in S_n}T_w \one_0 = \sum_{w \in S_n} t^{\ell(w)}\one_0 = W_0(t)\one_0\,,
\end{align}
where \begin{equation}
	W_0(t) = \sum_{w \in S_n } t^{\ell(w)}\,.
\end{equation}

Let $(\ZZ^n)_+ = \{\alpha \in \ZZ^n \,|\, \alpha_1 \geq \alpha_2 \geq \ldots \geq \alpha_n\}$. The symmetric group $S_n$ acts on $\ZZ^n$ by permuting co-ordinates. For $\lambda \in (\ZZ^n)_+$, let $S_{n,\lambda} = \{w \in S_n \,|\, w \lambda = \lambda \}$ and $S_n^\lambda$ be the set of minimal length coset representatives in $S_n/S_{n,\lambda}$. 

	Let $w = w^\lambda w_\lambda$ be the factorization of $w$ with $w^\lambda \in S_n^\lambda$ and $w_\lambda \in S_{n,\lambda}$, according to \autoref{Snfactorization}. Since $\ell(w) = \ell(w^\lambda) + \ell(w_\lambda)$, using \eqref{TuveqTuTv}, $$T_w = T_{w^\lambda}T_{w_\lambda} \quad \hbox{ for }\, w\in S_n\,.$$ Therefore, \begin{equation}
		\one_0 = \one^\lambda \one_\lambda \,, 
	\end{equation} where \begin{equation}\label{1lambdadef}
		\one^\lambda = \sum_{w \in S_n^\lambda}T_w \qquad \hbox{ and } \qquad \one_\lambda = \sum_{w \in S_{n,\lambda}}T_w \, .
	\end{equation}
Note that \begin{equation}\label{eq:1lambda10}
	\one_\lambda \one_0 = W_\lambda(t)\one_0\,, \qquad \hbox{ for } \lambda \in (\ZZ_{\geq 0}^n)_+\,,
\end{equation}
where $$ W_\lambda(t) = \sum_{w \in S_{n,\lambda}} t^{\ell(w)}\,. $$ 

\subsection{Commutation relations between the $T$ and $X$ in $H$}	
	Define for $\mu = (\mu_1,\ldots,\mu_n) \in \ZZ^n$, $X^\mu = X_1^{\mu_1}\ldots X_n^{\mu_n}$. Then by \eqref{XiXjcommutes}, $X^\mu X^\nu = X^{\mu+\nu} = X^{\nu}X^\mu$ for $\mu,\nu \in \ZZ^n$.
	
	\begin{lemma}{\sc\cite[(4.3.15)]{MacAHAbook}}\label{Lusztig}
		Let $\mu \in \ZZ^n$. Then \begin{align*}
			T_i X^\mu - X^{s_i \mu} T_i = \dfrac{t-1}{1-X^{\alpha_i}}(X^\mu - X^{s_i\mu}) \,.
		\end{align*}
	\end{lemma}
	\begin{proof}
		
%
		If the statement is true for $\mu, \nu \in \ZZ^n$ then it is true for $\mu+\nu$: \begin{align*}
			T_i X^{\mu +\nu} - X^{s_i(\mu+\nu)}T_i &= T_iX^{\mu+\nu}-X^{s_i\mu}T_iX^\nu + X^{s_i\mu}T_iX^\nu - X^{s_i(\mu+\nu)}T_i
			\\
			&= \dfrac{(t-1)(X^\mu-X^{s_i\mu})}{1-X^{\alpha_i}}X^\nu + X^{s_i\mu}\dfrac{(t-1)(X^\nu-X^{s_i\nu})}{1-X^{\alpha_i}}
			\\
			&= \dfrac{(t-1)(X^{\mu+\nu}-X^{s_i(\mu+\nu)})}{1-X^{\alpha_i}}\,.
		\end{align*}
		
		If the statement is true for $\mu \in \ZZ^n$ then it is true for $-\mu$: 
		\begin{align*}
			X^{s_i\mu}(T_iX^{-\mu} - X^{s_i(-\mu)}T_i)X^\mu &= X^{s_i\mu}T_i - T_iX^\mu = -\dfrac{t-1}{1-X^{\alpha_i}}(X^\mu - X^{s_i\mu})\,,
		\end{align*}
		which gives \begin{align*}
			T_iX^{-\mu} - X^{s_i(-\mu)}T_i = \dfrac{t-1}{1-X^{\alpha_i}}(X^{-\mu} - X^{s_i(-\mu)})\,.
		\end{align*}
		
		Therefore, it is enough to prove the statement for $\mu = \varepsilon_j$ for $j \in [n]$.	
		
		Suppose $j \notin \{i,i+1\}$. Then $s_i\varepsilon_j = \varepsilon_j$ and $T_iX_j - X_jT_i = 0$.
		
		Suppose $j = i$. Then \begin{align*}
			&T_iX_i - X_{i+1}T_i = T_iX_i - t^{-1}T_iX_iT_iT_i = T_iX_i - t^{-1}T_iX_i((t-1)T_i+t) 
			\\
			&=T_iX_i - t^{-1}(t-1)T_iX_iT_i - T_iX_i = -(t-1)X_{i+1} = \dfrac{t-1}{1-X^{\alpha_i}}(X_i-X_{i+1})\,.
		\end{align*}
		
		Suppose $j=i+1$. Then \begin{align*}
			T_iX_{i+1}-X_iT_i &= T_it^{-1}T_iX_iT_i - X_iT_i = t^{-1}((t-1)T_i+t)X_iT_i - X_iT_i 
			\\
			&= (t-1)X_{i+1} = \dfrac{t-1}{1-X^{\alpha_i}}(X_{i+1}-X_i)\,.
		\end{align*}
	\end{proof}
	
	\begin{corollary}\label{TiXmu}
		Let $\mu \in \ZZ^n$. Then 
			\item[(a)] if $(\mu , \alpha_i) = 1$ then $T_iX^\mu T_i = tX^{s_i\mu}$\,,
			\item[(b)] if $(\mu ,\alpha_i) = 0$ then $T_iX^\mu  = X^{\mu}T_i$\,.
	\end{corollary}
	
	\begin{proof}
		By \autoref{Lusztig}, if $(\mu,\alpha_i) \geq 0$ then
		$$ T_i X^\mu = X^{s_i\mu}T_i + (1-t)\sum_{j=1}^{(\mu,\alpha_i)} X^{\mu-j\alpha_i} \,.$$
		
			\item[(a)] Let $(\mu,\alpha_i) = 1$. Then $s_i\mu = \mu-(\mu,\alpha_i)\alpha_i = \mu - \alpha_i$\,. 
			
			Then using \eqref{tT_iinv}, \begin{align*}
				T_i X^\mu = X^{\mu- \alpha_i}(T_i+1-t) = X^{s_i\mu}tT_i^{-1}\,.
			\end{align*}
			
			\item[(b)] Let $(\mu,\alpha_i) = 0$. Then $s_i\mu = \mu$, so $$ T_iX^\mu = X^{s_i\mu}T_i \,.$$
		
	\end{proof}
	
	\begin{corollary}\label{TwXmucommutesifwfixesmu}
		Let $\mu \in \ZZ^n$ and $w \in S_{n}$ such that $w \mu = \mu$. Then $T_w X^\mu = X^\mu T_w$. 
	\end{corollary}
	
	\begin{proof}
		Since $S_{n,\mu}$ is a subgroup with generators $\{s_i \,|\, s_i \in S_{n,\mu} , i \in [n-1]\}$, the element $w \in S_{n,\mu}$ can be uniquely factored as a product of $s_i$s, each of which fixes $\mu$. Then use \autoref{TiXmu}[(b)].  
	\end{proof}

		\begin{proposition}\label{TuFXomegak}
			Let $I \in B(\varpi_k)$ for some $k\in [n]$. Then as elements of $H$, $$T_{u_I} X^{\varpi_k} = X^I t^{\ell(u_I)}(T_{u_I^{-1}})^{-1}\,.$$ 
		\end{proposition}
		
		\begin{proof}
			Let $u_I = s_{i_1}\ldots s_{i_\ell}$ be a reduced expression. By \autoref{omegak1}, $ (s_{i_{j+1}}\ldots s_{i_\ell} \varpi_k , \alpha_{i_j} ) = 1$ for every $j \in [\ell]$. Then applying \autoref{TiXmu}, $T_{i_j}X^{s_{i_{j+1}}\ldots s_{i_\ell} \varpi_k} = X^{s_{i_j}s_{i_{j+1}}\ldots s_{i_\ell} \varpi_k} tT_{i_j}^{-1}$. Then \begin{align*}
				T_{u_I}X^{\varpi_k} = T_{i_1}\ldots T_{i_\ell}X^{\varpi_k} = T_{i_1}\ldots T_{i_{\ell-1}}X^{s_{i_\ell}\varpi_k} tT_{i_\ell}^{-1} = \ldots = X^{s_{i_1}\ldots s_{i_\ell}\varpi_k} t^{\ell}(T_{u_I^{-1}})^{-1}\,.
			\end{align*}  
		\end{proof}
		
		
		Using \autoref{TuFXomegak} gives \begin{align}\label{10Xomegak}
			\one_0 X^{\varpi_k} &= \one^{\varpi_k}\one_{\varpi_k}X^{\varpi_k} = \one^{\varpi_k}X^{\varpi_k}\one_{\varpi_k} = \sum_{I \in B(\varpi_k)} X^{I}t^{\ell(u_I)}(T_{u_I^{-1}})^{-1}\one_{\varpi_k}\,.
		\end{align}

\begin{remark}
	The calculation \eqref{10Xomegak} is central to our work. Tracing back its proof, this uses \autoref{TuFXomegak}, which in turn uses \autoref{omegak1}. The statement of \autoref{omegak1} does not hold true for other types, it uses the fact that all the fundamental weights in type $GL_n$ are minuscule weights.
\end{remark}		
		
		\subsection{The Satake Isomorphism} 
		
%
		\begin{theorem}{\sc\cite[(4.2.6)]{MacAHAbook})}
			The affine Hecke algebra $H$ is a free $\ZZ[t^{\pm 1}]$-module with basis $\{ X^{\alpha}T_w \,|\, \alpha \in \ZZ^n \,, w \in S_n \}$\,. 
		\end{theorem}
		
		Since $T_w \one_0 = t^{\ell(w)}\one_0$, we have $H\one_0 \subseteq \ZZ[t^{\pm 1}][X_1^{\pm 1},\ldots,X_n^{\pm 1}]\one_0$. In fact, since for a polynomial $f(X) = f(X_1,\ldots,X_n) \in \ZZ[t^{\pm 1}][X_1,\ldots,X_n]$, we have $f(X)\one_0 \neq 0$, we can identify $f(X)$ with its image in $H\one_0$ as $f(X)\one_0$. So $$ H\one_0 = \ZZ[t^{\pm 1}][X_1^{\pm 1},\ldots,X_n^{\pm 1}]\one_0 \cong \ZZ[t^{\pm 1}][X_1^{\pm 1},\ldots,X_n^{\pm 1}]\,, $$ where the isomorphism is $f(X)\one_0 \mapsto f(X)$.
		
		\begin{theorem}{\sc\cite[Theorem 1.4, Theorem 2.4]{NelsenRam}} 
			The above map restricts to the following isomorphism. $$ \one_0 H \one_0 = Z(H)\one_0 = \ZZ[t^{\pm 1}][X_1^{\pm1},\ldots, X_n^{\pm 1}]^{S_n}\one_0 \cong \ZZ[t^{\pm 1}][X_1^{\pm1},\ldots, X_n^{\pm 1}]^{S_n} = Z(H) \,,$$ where $Z(H)$ denotes the center of the affine Hecke algebra $H$.
		\end{theorem}
		

	\section{The Hall-Littlewood Polynomials}

In this section we recall the definition of Hall-Littlewood polynomials and describe Macdonald's formula for the monomial expansion of these polynomials. We then rederive a recursive way to compute the monomial expansion, using ideas of Klostermann. 

\subsection{The Hall-Littlewood polynomials}

Recall the notations in \S\ref{subsec:parabolicsubgrquotient} and let \begin{equation}\label{Wlambdadef}
	W_\lambda(t) = \sum_{w \in S_{n,\lambda}}t^{\ell(w)} \,, \qquad \hbox{ for } \lambda \in (\ZZ^n)_+\,.
\end{equation}

The Hall-Littlewood $P$-polynomials are $P_\lambda(t) \in \ZZ[t^{\pm 1}][X_1^{\pm1},\ldots,X_n^{\pm1}]^{S_n}$ for $\lambda \in (\ZZ^n)_{+}$ defined by \begin{equation}\label{Plambdadef}
	P_\lambda(t) \one_0 = \dfrac{1}{W_\lambda(t)} \one_0 X^\lambda \one_0 \,,\qquad \hbox{ for } \lambda \in (\ZZ^n)_+ \,.
\end{equation}

For a proof that $P_\lambda(t) \in \ZZ[t^{\pm 1}][X_1^{\pm1},\ldots,X_n^{\pm1}]^{S_n}$ see \cite{NelsenRam}. 

Since $w \varpi_n = \varpi_n$ for every $w \in S_n$, by \autoref{TwXmucommutesifwfixesmu} we have $T_w X_1 \ldots X_n = X_1 \ldots X_n T_w$ for every $w \in S_n$. Therefore \begin{equation}
	P_{(\lambda_1+k,\ldots,\lambda_n+k)}(t) = (X_1\ldots X_n)^k P_\lambda(t) \,, \qquad \hbox{ for } \lambda \in (\ZZ^n)_+ \hbox{ and } k \in \ZZ\,. 
\end{equation} 

For the rest of the article we only focus on $P_\lambda(t)$ where $\lambda \in (\ZZ_{\geq 0}^n)_+$. In this case $P_\lambda(t) \in \ZZ[t][X_1,\ldots,X_n]^{S_n}$. For such $\lambda$s, Macdonald obtained a monomial expansion formula as a sum over semistandard Young tableaux, which we describe in the next subsections.

\subsection{Tableaux}\label{tableauxdef}

Let $\lambda \in (\ZZ_{\geq 0}^n)_+$. 

The \textit{Young diagram} of $\lambda$ is the upper-left justified arrangement of boxes with $\lambda_i$ many boxes in the $i$th row from top. We can assign a co-ordinate to each box of the diagram so that $(i,j)$ denotes the box in $i$th row from the top and $j$th column from the left.

The \textit{length} of $\lambda$, $\ell(\lambda)$, is the length of the first column of the diagram of $\lambda$.


A \textit{filling} of $\lambda$ is an assignment of numbers from $[n]$ to the boxes of the Young diagram of $\lambda$. Identifying $\lambda$ with the collection of boxes in its Young diagram, a filling $T$ of $\lambda$ is a function $T:\lambda \mapsto [n]$.

A \textit{semistandard Young tableau} of shape $\lambda$ is a filling of $\lambda$ such that the entries are weakly increasing along each row from left-to-right and the entries are strictly increasing along each column from top-to-bottom. The set of semistandard Young tableaux of shape $\lambda$ is denoted by $B(\lambda)$. 

\begin{example}
	Let $n = 3$. The semistandard Young tableaux of shape $(2,1,0)$ are \begin{align*}\ytableausetup{nosmalltableaux}
		\ytableaushort{11,2}\,, \quad \ytableaushort{11,3}\,, \quad \ytableaushort{12,2}\,, \quad \ytableaushort{12,3}\,,\quad \ytableaushort{13,2}\,,\quad\ytableaushort{13,3}\,,\quad \ytableaushort{22,3}\,,\quad \ytableaushort{23,3}\,.
	\end{align*}
\end{example}



Suppose $\lambda, \mu \in (\ZZ_{\geq 0}^n)_+$ are such that the smallest column length of $\lambda$ is at-least as big as the biggest column length of $\mu$. In other words, if $\lambda = \varpi_{a_p}+\ldots + \varpi_{a_1}$ and $\mu = \varpi_{b_q} + \ldots + \varpi_{b_1}$ with $n \geq a_{p} \geq \ldots \geq a_1 \geq 1$ and $n \geq b_{q} \geq \ldots  \geq b_1 \geq 1$ then $a_1 \geq b_{q}$. If $T_1$ and $T_2$ are fillings of $\lambda$ and $\mu$ respectively then we write $T_1 \otimes T_2$ to denote the filling of $\lambda+\mu$ obtained by filling the first $p$ columns by $T_1$ and the next $q$ columns by $T_2$. For $\lambda, \mu$ satisfying the condition above let \begin{equation}
	B(\lambda) \otimes B(\mu) = \{ T_1 \otimes T_2 \,| \, T_1 \in B(\lambda) \hbox{ and } T_2 \in B(\mu) \}\,.
\end{equation} 

Note that even if $T_1,T_2$ are semistandard, $T_1 \otimes T_2$ may not be semistandard.

\begin{example}
	Let $T_1 = \ytableaushort{1,3,4}$\, and \, $T_2 = \ytableaushort{1,2}$\,. Then $T_1 \otimes T_2 = \ytableaushort{11,32,4}$\,.
\end{example}

Let $\lambda = \varpi_{a_p} + \ldots + \varpi_{a_1}$, with $n \geq a_p \geq \ldots \geq a_1 \geq 1$. With the above notations, we can write \begin{equation}
	B(\lambda) \subseteq B(\varpi_{a_p}) \otimes \ldots \otimes B(\varpi_{a_1}) \,.
\end{equation}

\subsection{Macdonald's formula}\label{sec:Macformula}

For two elements $\mu, \lambda \in (\ZZ_{\geq 0}^n)_+$ write $\mu \subseteq \lambda$ if $\lambda_i \geq \mu_i$ for all $i \in [n]$. The set theoretic difference of the Young diagrams $\lambda - \mu$ is called a skew diagram. If $\lambda_1 \geq \mu_1 \geq \lambda_2 \geq \mu_2 \geq \ldots $ then we call $\lambda - \mu$ a horizontal strip.

A semistandard tableau $T$ of shape $\lambda$ may be written as a flag of partitions $\GT(T) = (T_{\leq 1} \subseteq T_{\leq 2} \subseteq \ldots \subseteq T_{\leq n})$ where $T_{\leq i}$ is the shape with fillings $\leq i$. Then $T_{\leq i+1} - T_{\leq i}$ is a horizontal strip for each $i \in [n-1]$.

Let $\theta = \lambda - \mu$ be a horizontal strip. Let $$J = \{j\in \ZZ_{\geq 1}\,|\, \theta'_j < \theta'_{j+1} \} = \{j\in \ZZ_{\geq 1}\,|\, \theta'_j = 0, \theta'_{j+1}=1 \}\,,$$ where $\theta'_i$ denotes the length of the $i$th column of $\theta$.
Then define $$ \psi_{\lambda/\mu} = \prod_{j \in J}(1-t^{m_j(\mu)}) \,,$$ where $m_j(\mu) = \# \{ i \in [n]\,|\, \mu_i = j \}$.

Let $T$ be a semi-standard Young tableau of shape $\lambda$, with corresponding sequence $\GT(T)$ as above. Let $$ \psi_T = \prod_{i = 1}^{n-1} \psi_{T_{\leq i+1}/T_{\leq i}} \,.$$ 


For a tableau $T:\lambda \to [n]$, define $X^T = \prod\limits_{b \in \lambda} X_{T(b)} = \prod\limits_{i = 1}^{n} X_i^{c_i}$, where $c_i = \# T^{-1}(i)$, i.e, the number of $i$ in $T$. 

Then \begin{theorem}{\sc\cite[Chapter III, (5.11') ]{MacMainBook}}\label{ThMacformula}
	\begin{equation}
		 P_\lambda(t) = \sum_{T \in B(\lambda)} \psi_T X^T \,,\qquad \hbox{ for } \lambda \in (\ZZ_{\geq 0}^n)_+ \,.
	\end{equation}
\end{theorem}
\ytableausetup{smalltableaux}
\begin{example}
	Let $n = 3$, $\lambda = (2,1,0)$. Then the coefficient of $X_1X_2X_3$ in $P_{(2,1,0)}$ is $$\psi_{\ytableaushort{12,3}} + \psi_{\ytableaushort{13,2}}\,,$$
	where $$ \psi_{\ytableaushort{12,3}} = \psi_{\ydiagram{2}*[*(gray)]{1+1}} \cdot \psi_{\ydiagram{2,1}*[*(gray)]{0,1}} = (1-t)\cdot 1 \,, \qquad \hbox{ and } \qquad \psi_{\ytableaushort{13,2}} = \psi_{\ydiagram{1,1}*[*(gray)]{0,1}} \cdot \psi_{\ydiagram{2,1}*[*(gray)]{1+1}} = 1\cdot (1-t^2)\,. $$

\end{example}

\begin{remark}
	Note that the above example shows that even though $P_{\lambda}$ is symmetric the $\psi_T$s are not directly symmetric, i.e, $\psi_T$ may not equal $\psi_{wT}$ for $w \in S_n$ where $wT$ denotes the crystal action of $w$ on $T$.
\end{remark}

Klostermann \cite{Klostermann} obtained a recursive formula describing $\psi_T$. We rederive these recursions in \autoref{prop:klostermannrecs}. We then obtain a Hecke algebra lift of Klostermann's recursions. This provides another proof of Macdonald's formula. 

The next two subsections are for obtaining the recursions for $\psi_T$.


\subsection{Writing $\psi_T$ box-by-box and column-by-column}

For a box $b$ in the diagram of $\lambda$ we write $b \in \lambda$. For a box $b \in \lambda$, let $\Leg(b,\lambda)$ denote the set of all boxes below $b$ in the same column.


Let $T \in B(\lambda)$.  Let $\leg_T(b,<i) = \# \{ b' \in \Leg(b,\lambda)\,|\, T(b')<i \}$. 

%


For a box $b = (r,c) \in \lambda$ define \begin{equation}
	\psi_T(b) = \begin{cases}
	1- t^{\leg_T(b,<i)+1}\,, & \hbox{ if } T(r,c+1) = i \hbox{ and } i \notin T(\Leg_\lambda(b))\cup T(b),
	\\
	1\,, & \hbox{ otherwise}.
\end{cases}
\end{equation}

Then \begin{equation}
	\psi_T = \prod_{b \in \lambda} \psi_T(b) \,.
\end{equation}


\subsection{Klostermann recursions}

Recall the action of $S_n$ on the set of all columns from \S\ref{columndef}. 

%
%
%

The following recursions were first described by Klostermann in the proof of proposition 27 in \cite{Klostermann}.

\begin{proposition}
\label{prop:klostermannrecs}
	\begin{enumerate}
		\item Suppose $1 \leq \ell_1 \leq \ldots \leq \ell_r \leq n$. If $T = C_r \otimes \ldots \otimes C_1$ with $C_i \in B(\varpi_{\ell_i})$ then \begin{equation}\label{eq:psiTproduct}
			\psi_T = \prod_{j = 1}^{r-1} \psi_{C_{j+1}\otimes C_{j}} \,.
		\end{equation}
		
		\item Let $1 \leq a \leq b \leq n$. For $F \in B(\varpi_b), E \in B(\varpi_a)$, the $\psi_{F\otimes E}$ are determined by the following recursions. \begin{enumerate}
			\item[(a)] If $F\otimes E$ is not semistandard then $\psi_{F\otimes E} = 0$,
			\item[(b)] If $E = (1,\ldots,a)$ and $F\otimes E$ is semistandard then $\psi_{F\otimes E} = 1$.
			
			\item[(c)] Suppose $F\otimes E$ is semistandard. If $j \notin E, j+1 \in E$ then \begin{align*}
				\psi_{F\otimes E} = \begin{cases}
					t \psi_{s_jF\otimes s_jE} + (1-t) \psi_{F\otimes s_jE}\,, & \hbox{ if } j \in F, j+1 \notin F\,,
					\\
					\psi_{s_jF\otimes s_jE}\,, &\hbox{ otherwise} \,.
				\end{cases}
			\end{align*}
		\end{enumerate}
		
	\end{enumerate}	
\end{proposition}

\begin{proof}
		\item[(1)] Since each $\psi_T(b)$ depends on only the column of $b$ and the column right next to it, \eqref{eq:psiTproduct} follows.
 
		\item[(2.a)] is part of the definition of $\psi_{F\otimes E}$.
		
		\item[(2.b)] Since $F\otimes E$ is semistandard, $f_1 = 1,\ldots,f_a = a$. So the first $a$ entries of $F$ is the same as the corresponding entries in $E$, hence $\psi_{F\otimes E} = 1$.
		
		\item[(2.c)] 
		Let $\leftarrow k$ denote the box in $F$ immediate left of the box in $E$ containing $k$, and 	
		let $\psi_{F\otimes E}(\leftarrow k)$ denote the  contribution from the box in $F$ immediate left of the box in $E$ containing $k$.

\noindent		Suppose $j \notin E, j+1 \in E$.
		
		\textbf{Case 1:} Suppose $j \in F, j+1 \notin F$. If $d = \leg(\leftarrow j+1,<j+1)+1 > 1$  then $\psi_{F\otimes E}(\leftarrow j+1) = 1-t^d, \psi_{s_jF\otimes s_jE}(\leftarrow j) = 1-t^{d-1}, \psi_{F\otimes s_jE}(\leftarrow j) = 1$.
			
			\ytableausetup{boxsize=2.5em}
			$$F\otimes E = \ytableaushort{.{j+1},{\vdots},j}\,, \quad s_jF\otimes s_jE = \ytableaushort{.{j},{\vdots},{j+1}}\,, \quad F\otimes s_jE = \ytableaushort{.j,{\vdots},j}\,.$$
			And $\psi_{F\otimes E}(\leftarrow k) = \psi_{s_jF\otimes s_jE}(\leftarrow k) = \psi_{F\otimes s_jE}(\leftarrow k)$ for $k \notin \{j,j+1\}$.
			Then since $(1-t^d) = t(1-t^{d-1})+(1-t)1$ we get $\psi_{F\otimes E} = t\psi_{s_jF\otimes s_jE}+(1-t)\psi_{F\otimes s_jE}$.
			
			If $d = 1$ then $s_jF \otimes s_jE$ is not semistandard, so $\psi_{s_jF\otimes s_jE} = 0$ and $\psi_{F\otimes E}(\leftarrow j+1) = 1-t, \psi_{F\otimes s_jE}(\leftarrow j) = 1$. 
			
			\ytableausetup{boxsize=2.5em}
			$$F\otimes E = \ytableaushort{j{j+1}}\,, \quad s_jF\otimes s_jE = \ytableaushort{{j+1}{j}}\,, \quad F\otimes s_jE = \ytableaushort{jj}\,.$$
				
		\noindent	In this case the statement is true since $(1-t^1) = t\cdot 0 + (1-t)\cdot 1$.			
			
			\vspace{1em}
			
			\textbf{Case 2:} Suppose $j,j+1 \in F$. Note that since $F\otimes E$ is semistandard, $j+1$ of $E$ is above or in the same row of $j+1$ of $F$. If $j+1$ of $E$ is in the same row as $j+1$ of $F$ then by semistandardness of $F \otimes E$, $j \in E$ also. 
			
			If $j+1$ of $E$ is in a row strictly above $j+1$ of $F$ then $\psi_{F\otimes E}(\leftarrow j+1) = 1 = \psi_{s_jF\otimes s_jE}(\leftarrow j)$, and $\psi_{F\otimes E}(\leftarrow k) = \psi_{s_jF\otimes s_jE}(\leftarrow k)$ for $k \notin \{j,j+1\}$. So $\psi_{F\otimes E} = \psi_{s_jF\otimes s_jE}$.
			
			$$F\otimes E = \ytableaushort{.{j+1},{\vdots},j,{j+1}}\,, \quad s_jF\otimes s_jE = \ytableaushort{.{j},{\vdots},j,{j+1}}\,.$$
			
%
%
			\textbf{Case 3:} If $j, j+1 \notin F$ then $\psi_{F\otimes E}(\leftarrow j+1) = 1-t^d = \psi_{s_jF\otimes s_jE}(\leftarrow j)$, where $d = \leg(\leftarrow j+1,<j+1)+1$, and $\psi_{F\otimes E}(\leftarrow k) = \psi_{s_jF\otimes s_jE}(\leftarrow k)$ for $k \neq j+1$. So $\psi_{F\otimes E} = \psi_{s_jF\otimes s_jE}$.
			
			$$F\otimes E = \ytableaushort{.{j+1},{\vdots},.}\,, \quad s_jF\otimes s_jE = \ytableaushort{.{j},{\vdots},.}\,.$$

			\textbf{Case 4:} If $j \notin F, j+1 \in F$ then $\psi_{F\otimes E}(\leftarrow j+1) = 1, \psi_{s_jF\otimes s_jE}(\leftarrow j) = 1$, and $\psi_{F\otimes E}(\leftarrow k) = \psi_{s_jF\otimes s_jE}(\leftarrow k)$ for $k \notin \{j,j+1\}$. So $\psi_{F\otimes E} = \psi_{s_jF\otimes s_jE}$.
			
			$$F\otimes E = \ytableaushort{.{j+1},{\vdots},{j+1}}\,, \quad s_jF\otimes s_jE = \ytableaushort{.{j},{\vdots},j}\,.$$

\end{proof}

\begin{example}
	\ytableausetup{smalltableaux}
	\begin{align*}
		&\psi_{\ytableaushort{12,3}} = t \psi_{\ytableaushort{21,3}} + (1-t) \psi_{\ytableaushort{11,3}} = t\cdot 0 + (1-t) \cdot 1 = 1-t \,,
		\\
		&\psi_{\ytableaushort{12,2}} = \psi_{\ytableaushort{11,2}} = 1\,,
		\\
		&\psi_{\ytableaushort{13,2}} = t \psi_{\ytableaushort{12,3}} + (1-t) \psi_{\ytableaushort{12,2}} = t\cdot (1-t) + (1-t) \cdot 1 = 1-t^2\,.
	\end{align*}
\end{example}

\section{AHA lift of $\psi_T$}

This section defines our main objects of study, the $\Psi_T \in H_n$, where $T$ is a column-strict filling. Our main goal in this section is to establish \autoref{10Xlambda}, which gives an affine Hecke algebra lift of the monomial expansion formula of Macdonald \autoref{ThMacformula}. 

\subsection{Parabolic Expansions}

Let $H_n$ be the $\ZZ[t^{\pm1}]$ subalgebra of $H$ generated by $T_1,\ldots,T_{n-1}$. Then $H_n$ has $\ZZ[t^{\pm1}]$-basis $\{ T_w \,|\, w \in S_n \}$. Let $H_{n,\lambda}$ be the subalgebra of $H_n$ with basis $\{T_w \,|\, w \in S_{n,\lambda}\}$. In particular, for $\ell \in [n]$, $H_{n,\varpi_{\ell}}$ is the subalgebra of $H_n$ generated by $T_1,\ldots,T_{\ell-1},T_{\ell+1},\ldots,T_{n-1}$. Hence, $H_{n,\varpi_{\ell}} \cong H_{\ell} \times H_{n-\ell}$.

\begin{lemma}\label{Hnparabolicdecomp}
	Let $h \in H_n$ and $\ell \in [n]$. There exists unique decomposition  $$h = \sum_{F \in B(\varpi_{\ell})} T_{u_F}h_F \,, \qquad \hbox{ with } h_F \in H_{n,\varpi_{\ell}}\,.$$ 
\end{lemma}
\begin{proof}
	Since $H_n$ has $\ZZ[t^{\pm 1}]$ basis $\{T_w \,|\, w \in S_n\}$, it is enough to decompose each $T_w$. From \autoref{Snfactorization}, $w$ has a unique factorization $w = uv$ with $u \in S_n^{\varpi_{\ell}}$ and $v\in S_{n,\varpi_{\ell}}$, with $\ell(w) = \ell(u)+ \ell(v)$. By \autoref{Snparabolicquotients}, $u = u_F$ for some $F \in B(\varpi_{\ell})$.  Then by \eqref{TuveqTuTv}, $T_w = T_{u_F} T_v$, and $T_v \in H_{n,\varpi_{\ell}}$.
\end{proof}


The following proposition follows from the subword property of the Bruhat order and the quadratic relation \eqref{quadTi} in the Hecke algebra. For a complete proof see \cite[Lemma 3.6]{IntrotoSoergelBimodules}.

\begin{proposition}\label{prop:Rpolysupport}
	There exists polynomials $a_v(t) \in \ZZ[t^{\pm1}]$ such that $$(T_{w^{-1}})^{-1} = \sum\limits_{v \leq w }a_{v}(t) T_v\,.$$. In other words, $(T_{w^{-1}})^{-1}$ is a linear combination of $T_v$ with $v \leq w$.
\end{proposition}


\begin{proposition}\label{semistandardness2}
	Let $h \in H_{n,\varpi_{a}}$. Let $E$ be a column of length $a$ and let $b \geq a$. Write \begin{equation}\label{eq:semistandarness1}
		t^{\ell(u_E)}(T_{u_E^{-1}})^{-1} h = \sum_{F \in B(\varpi_{b})} T_{u_F}h_F \,, \qquad \hbox{ with } h_F \in H_{n,\varpi_b}\,.
	\end{equation} If $F\otimes E$ is not semistandard then $h_F = 0$.
\end{proposition}

\begin{proof}

Since for any $k \in [n]$, $S_{n,\varpi_k}$ is a group, and the subalgebra $H_{n,\varpi_k}$ of $H_n$ is invariant under taking inverses, it is enough to prove the proposition for $t^{\ell(u_E)}(T_{u_E^{-1}})^{-1}(T_{v^{-1}})^{-1}$ for some $v \in S_{n,\varpi_a}$. Let $w = u_E v$, then, $ (T_{u_E^{-1}})^{-1}(T_{v^{-1}})^{-1} = (T_{w^{-1}})^{-1}$. Using \autoref{prop:Rpolysupport}, $(T_{w^{-1}})^{-1}$ is a linear combination of $T_x$ with $x \leq w$. Now, write $x = u_F z$ for $u_F \in S_{n}^{\varpi_b}$ and $z \in S_{n,\varpi_b}$, according to \autoref{Snfactorization}. Then $T_x = T_{u_F}T_z$. Therefore, the terms appearing in $(T_{w^{-1}})^{-1}$ correspond to $u_F$ for $u_F \leq w = u_E v$. Then by \autoref{bruhattableaucondition}, $u_F[a]_{\leq} \leq w[a]_{\leq} = E$, since $v$ fixes $[a]$ and $u_E[a]_{\leq} = E$. Let $u_F[a]_{\leq} = (f_1',\ldots,f_a')$. Since $u_F[b]_{\leq} = F = (f_1,\ldots,f_b)$, and $b \geq a$, we have $f_i \leq f'_i$ for $i \in [a]$. Hence $F\otimes E$ is semistandard.
\end{proof}

\begin{example}
	Let $n = 3$, $a = 1,b=2$. Then $S_{3,\varpi_1} = S_{[1]}\times S_{[2,3]} = \{1,s_2\}$, $$S_{3}^{\varpi_1} =\{ u_{\ytableaushort{1}} = \begin{pmatrix}
		1 & 2 & 3 \\ 1 & 2 & 3
	\end{pmatrix} , u_{\ytableaushort{2}} = \begin{pmatrix}
		1 & 2 & 3 \\ 2 & 1 & 3
	\end{pmatrix} = s_1 , u_{\ytableaushort{3}} = \begin{pmatrix}
		1 & 2 & 3 \\ 3 & 1 & 2
	\end{pmatrix} = s_2s_1 \}.$$
	And $S_{3,\varpi_2} = S_{[1,2]}\times S_{[3]} = \{1,s_1\}$, $$ S_3^{\varpi_2} = \{ u_{\ytableaushort{1,2}} = \begin{pmatrix}
		1 & 2 & 3 \\ 1 & 2 & 3
	\end{pmatrix} = 1, u_{\ytableaushort{1,3}} = \begin{pmatrix}
		1 & 2 & 3 \\ 1 & 3 & 2
	\end{pmatrix} = s_2, u_{\ytableaushort{2,3}} = \begin{pmatrix}
		1 & 2 & 3 \\ 2 & 3 & 1
	\end{pmatrix} = s_1s_2   \}.$$
	
	\begin{align*}
		&t^{\ell(u_{\ytableaushort{1}})}(T_{u_{\ytableaushort{1}}^{-1}})^{-1} = 1 \in H_{3,\varpi_2}\,,
		\\
		&t^{\ell(u_{\ytableaushort{2}})}(T_{u_{\ytableaushort{2}}^{-1}})^{-1} = tT_1^{-1} = 1\cdot(T_1+1-t) \in H_{3,\varpi_2}\,,
		\\
		&t^{\ell(u_{\ytableaushort{3}})}(T_{u_{\ytableaushort{3}}^{-1}})^{-1} = tT_2^{-1}tT_1^{-1} = (T_2+1-t)tT_1^{-1} = T_2\cdot tT_1^{-1} + 1\cdot (1-t)tT_1^{-1} 
		\\
		&\qquad \qquad \qquad= T_{u_{\ytableaushort{1,3}}}\cdot tT_1^{-1} + T_{u_{\ytableaushort{1,2}}} \cdot (1-t)tT_1^{-1}\,. 
	\end{align*}
\end{example}

In particular, when $E = (1,\ldots,a)$ we get the following corollary.

\begin{corollary}\label{Hnparabolicdecomp2}
	Let $1 \leq a \leq b \leq n$. Suppose $h \in H_{n,\varpi_a}$. Then write $$ h = \sum_{F \in B(\varpi_{b})} T_{u_F}h_F \,, \qquad \hbox{ with } h_F \in H_{n,\varpi_{b}}\,.$$ Then $h_F \neq 0$ only if $f_1 = 1,\ldots, f_a = a$. 
\end{corollary}

\subsection{Definition of $\Psi_T$} 

%
%


%

\vspace{0.25cm}
Let $T \in B(\varpi_{\ell_r})\otimes \ldots \otimes B(\varpi_{\ell_1})$ with $1 \leq \ell_1 \leq 
\ldots \leq \ell_r \leq n$. We will now give a recursive definition of an element $\Psi_T$ for each such $T$ 

	For a column $C \in B(\varpi_\ell)$, define \begin{equation}\label{Psi_C}
	\Psi_C = t^{\ell(u_C)}(T_{u_C^{-1}})^{-1} \one_{\varpi_\ell} . 
\end{equation}

Suppose $T = C \otimes S$ with $C \in B(\varpi_\ell)$ and $S \in B(\varpi_{\ell_r}) \otimes \ldots \otimes B(\varpi_{\ell_1})$, where $1 \leq \ell_1 \leq \ldots \leq \ell_r \leq \ell \leq n$.   

\noindent By \autoref{Hnparabolicdecomp}, write $$ \Psi_S = \sum_{E \in B(\varpi_{\ell})} T_{u_E}h_{E,S} \qquad \hbox{ with } \qquad h_{E,S} \in H_{n,\varpi_{\ell}} \,.$$
	
\noindent Then define  \begin{equation}
	\Psi_T = t^{\ell(u_C)}(T_{u_C^{-1}})^{-1}h_{C,S} \,.
\end{equation} 

%
%
\noindent
This defines $\Psi_T \in H_n$ for every $T$. 

The following theorem says that the function $\Psi$ picks up the Cartan component (i.e, the semistandard Young tableaux) from tensor product of column crystals. 

\begin{theorem}\label{10Xlambda}
	Let $T \in B(\varpi_{\ell_r})\otimes \ldots \otimes B(\varpi_{\ell_1})$ with $1 \leq \ell_1 \leq 
	\ldots \leq \ell_r \leq n$. \begin{enumerate}
		\item[(a)]\label{PsiT0ifnotSSYT}  If $T$ is not semistandard then $\Psi_T = 0$.
		\item[(b)] For $\ell \in [n]$,  \begin{equation}\label{PsiXomega}
			\Psi_{T} X^{\varpi_{\ell}} = \sum_{C \in B(\varpi_{\ell})} X^C \Psi_{C\otimes T} \,.
		\end{equation}
		\item[(c)] Let $\lambda \in (\ZZ_{\geq 0}^n)_+$.  Then \begin{equation}\label{eq:10Xlambda}
			 \one_0 X^{\lambda} = \sum_{T \in B(\lambda)} X^{T} \Psi_{T}\,.
		\end{equation} 
	\end{enumerate}
\end{theorem}

\begin{proof}
		\item[(a)]  
		If $T$ has only one column then $T$ is semistandard, so the statement is true. 
		
		Suppose $T$ has two columns. Write $T = F \otimes E$ with $F \in B(\varpi_{b})$ and $E \in B(\varpi_a)$. Using the definition of $\Psi$ and \autoref{Hnparabolicdecomp} write $$ \Psi_E = t^{\ell(u_E)}(T_{u_E^{-1}})^{-1} \one_{\varpi_a} = \sum_{C \in B(\varpi_b)}T_{u_C}h_{C,E} \quad \hbox{ with } \quad h_{C,E} \in H_{n,\varpi_b} \,.$$
		\noindent
		Taking $h = \one_{\varpi_a}$ in \autoref{semistandardness2}, if $F \otimes E$ is not semistandard then $h_{F,E} = 0$, so $$\Psi_{T} = t^{\ell(u_F)}(T_{u_F^{-1}})^{-1} h_{F,E} = 0\,.$$
		
		Assume now that the statement is true for $r \geq 2$ columns, i.e, $\Psi_S = 0$ if $S$ is not semistandard and $S$ has $\leq r$ columns. We will prove the statement for $r+1$ columns.  
		
		\noindent Let $T = C \otimes C_r \ldots \otimes C_1$ with $C \in B(\varpi_\ell)$, $C_i \in B(\varpi_{\ell_i})$ for $i \in [r]$,  and $1 \leq \ell_1 \leq \ldots \leq \ell_r \leq \ell \leq n$. Let $S = C_r \otimes \ldots \otimes C_1$ and $S' = C_{r-1}\otimes \ldots \otimes C_1$. 
		
		\noindent Using the definition of $\Psi$ and \autoref{Hnparabolicdecomp}, write
		$$ \Psi_{S} = t^{\ell(u_{C_r})}(T_{u_{C_r}^{-1}})^{-1} h_{C_{r},S'} = \sum_{C \in B(\varpi_{\ell})} T_{u_C} h_{C,S} \quad \hbox{ with } \quad h_{C,S} \in H_{n,\varpi_{\ell}}\,. $$
		
		\noindent Taking $h = h_{C_r,S}$, and $E = C_r$ in \autoref{semistandardness2}, if $C\otimes C_r$ is not semistandard then $h_{C,S} = 0,$ so $$\Psi_{T} = t^{\ell(u_C)}(T_{u_C^{-1}})^{-1}h_{C,S} = 0\,.$$
		
		\item[(b)]  
		Write $$ \Psi_T = \sum_{C \in B(\varpi_{\ell})} T_{u_C}h_{C,T} \qquad \hbox{ with } h_{C,T} \in H_{n,\varpi_{\ell}} \,.$$
		
		\noindent By \autoref{TwXmucommutesifwfixesmu}, $$h_{C,T}X^{\varpi_{\ell}} = X^{\varpi_{\ell}} h_{C,T}\,.$$
		
		\noindent Using \autoref{TuFXomegak},
		\begin{align*}
			\Psi_{T} X^{\varpi_{\ell}} &= \sum_{ C \in B(\varpi_{\ell}) } T_{u_C}h_{C,T} X^{\varpi_{\ell}} = \sum_{ C \in B(\varpi_{\ell}) } T_{u_C}X^{\varpi_{\ell}}h_{C,T} \\&= \sum_{ C \in B(\varpi_{\ell}) } X^C t^{\ell(u_C)}(T_{u_C^{-1}})^{-1}h_{C,T}  = \sum_{ C \in B(\varpi_{\ell}) } X^C \Psi_{C\otimes T} \,.
		\end{align*}
		
	\item[(c)] 
	Let $\lambda = \varpi_{\ell_1}+\ldots+\varpi_{\ell_k}$ with $1\leq \ell_1 \leq \ldots \leq \ell_k \leq n$. Then by \eqref{10Xomegak} and \eqref{Psi_C}, $$ \one_0 X^{\varpi_{\ell_1}} = \sum_{C \in B(\varpi_{\ell_1})} X^C t^{\ell(u_C)}(T_{u_C^{-1}})^{-1} = \sum_{C \in B(\varpi_{\ell_1})} X^C \Psi_C \,.$$
		
		\noindent Let $j \in [k]$ and $\lambda^{(j)} = \varpi_{\ell_1}+ \ldots + \varpi_{\ell_j} $. Suppose that $$ \one_0 X^{\lambda^{(j)}} = \sum_{S \in B(\lambda^{(j)})} X^S \Psi_S \,. $$
		
		\noindent Then by \eqref{PsiXomega},
		\interdisplaylinepenalty=10000 \begin{align*}
			\one_0 X^{\lambda^{(j)}+\varpi_{\ell_{j+1}}} &= \sum_{S \in B(\lambda^{(j)})} X^S \Psi_S X^{\varpi_{\ell_{j+1}}} \\&= \sum_{S \in B(\lambda^{(j)})} X^S \sum_{\substack{C \in B(\varpi_{\ell_{j+1}}) \\ C\otimes S \text{ semistandard }}} X^C \Psi_{C\otimes S} = \sum_{T \in B(\lambda^{(j+1)})} X^T \Psi_T \,.
		\end{align*} \interdisplaylinepenalty=1000
		
	\noindent	This proves the statement by induction.

\end{proof}
\begin{example}
	Let $n = 3$. For $T \in B((2,1,0))$, we compute $\Psi_T$.
	
	Note that $\one_{\varpi_1} = 1+T_2$.
	
	Since \begin{align*}
		\Psi_{\ytableaushort{1}} = t^{\ell(u_{\ytableaushort{1}})}(T_{u_{\ytableaushort{1}}^{-1}})^{-1}\one_{\varpi_1} = 1+T_2 = T_{u_{\ytableaushort{1,2}}} + T_{u_{\ytableaushort{1,3}}}\,,
	\end{align*}
	then \begin{align*}
		\Psi_{\ytableaushort{11,2}} = 1 \,, \qquad \Psi_{\ytableaushort{11,3}} = tT_2^{-1}\,.
	\end{align*}
	
	Since \begin{align*}
		\Psi_{\ytableaushort{2}} &= tT_1^{-1}(1+T_2) = (T_1+1-t)(1+T_2) = 1\cdot tT_1^{-1} + T_2  \cdot (1-t) + T_1T_2 
		\\
		&= T_{u_{\ytableaushort{1,2}}} \cdot tT_1^{-1} + T_{u_{\ytableaushort{1,3}}}\cdot (1-t) + T_{u_{\ytableaushort{2,3}}}\,,
	\end{align*}
	then \begin{align*}
		\Psi_{\ytableaushort{12,2}} = tT_1^{-1} \,, \qquad \Psi_{\ytableaushort{12,3}} = tT_2^{-1}(1-t) \,, \qquad \Psi_{\ytableaushort{22,3}} = t^2T_1^{-1}T_2^{-1}\,.
	\end{align*}
	
	Since \begin{align*}
		\Psi_{\ytableaushort{3}} &= tT_2^{-1}tT_1^{-1}(1+T_2) = tT_2^{-1}tT_1^{-1}(t+tT_2^{-1}) = t^3T_2^{-1}T_1^{-1} + t^3T_2^{-1}T_1^{-1}T_2^{-1}
		\\
		&= t^3T_2^{-1}T_1^{-1} + t^3T_1^{-1}T_2^{-1}T_1^{-1} = (t+tT_1^{-1})t^2T_2^{-1}T_1^{-1} = (1+T_2)tT_2^{-1}tT_1^{-1} 
		\\
		&= (1+T_1)(T_2+1-t)tT_1^{-1} = 1 \cdot (1-t)(1+T_1)tT_1^{-1} + T_2 \cdot tT_1^{-1} + T_1T_2 \cdot tT_1^{-1}
		\\
		&= T_{u_{\ytableaushort{1,2}}}\cdot (1-t)(1+T_1)tT_1^{-1} + T_{u_{\ytableaushort{1,3}}}\cdot tT_1^{-1} + T_{u_{\ytableaushort{2,3}}}\cdot tT_1^{-1} \,,
	\end{align*}
	then \begin{align*}
		\Psi_{\ytableaushort{13,2}} = (1-t)(1+T_1)tT_1^{-1} = (1-t)(tT_1^{-1}+t) = (1-t)(1+T_1) \,,\\ \qquad \Psi_{\ytableaushort{13,3}} = tT_2^{-1}\cdot tT_1^{-1} \,,\qquad \Psi_{\ytableaushort{23,3}} = t^2T_1^{-1}T_2^{-1}\cdot tT_1^{-1}\,.
	\end{align*}
\end{example}

\section{Recursions for $\Psi_T$}
This section presents a recursive formula to compute the elements $\Psi_T \in H_n$, with an initial condition. The recursion aims to lower the rightmost column of $T$ which is not highest weight, thus by repeated application of the recursion one can compute $\Psi_T$ from that of a highest weight tableau.

\subsection{Initial Conditions}
A column $C\in B(\varpi_{\ell})$ of length $\ell$ is called \textit{highest weight} if $C = (1,\ldots,\ell)$. A tableau $T$ is called \textit{highest weight} if all its columns are highest weight. 

For $\lambda \in (\ZZ_{\geq 0}^n)_+$, there is a unique highest weight tableau $T^0_\lambda$ of shape $\lambda$, whose all entries in the $i$th row is $i$.

\begin{example}
	For $n = 3$ and $\lambda = (2,1,0)$, $$\ytableausetup{nosmalltableaux, boxsize=normal} T^0_{(2,1,0)} = \ytableaushort{11,2} \,.$$
\end{example}

\begin{remark}
	The highest weight column in $B(\varpi_\ell)$ is actually the lowest element in the poset $B(\varpi_\ell)$ with the partial order $\leq$ on columns from \eqref{eq:colposetdef}. The terminology highest weight comes from crystal basis theory, since the highest weight tableaux are the highest weight elements in the crystal of tableaux.
\end{remark}

Recall that for $\lambda \in (\ZZ_{\geq 0}^n)_+$, the element $\one_\lambda$ was defined in \eqref{1lambdadef}. For $1 \leq a \leq b \leq n$ define $$ \one_{[a,b]} = \sum_{w \in S_{[a,b]}} T_w \,. $$ If $\lambda = \varpi_{\ell_1} + \ldots + \varpi_{\ell_r}$ with $1 \leq \ell_1 \leq \ldots \leq \ell_r \leq n$ then $S_{n,\lambda} = S_{[1,\ell_1]} \times S_{[\ell_1+1,\ell_2]}\times \ldots\times S_{[\ell_r+1,n]}$. It follows that $$ \one_\lambda = \one_{[1,\ell_1]} \one_{[\ell_1+1,\ell_2]} \ldots \one_{[\ell_r+1,n]} \,.$$ 

\begin{proposition}\label{Psiinitialcondition}
	Let $T^0_\mu$ be the highest weight tableau of shape $\mu \in (\ZZ_{\geq 0}^n)_+$. Let $\ell \geq \ell(\mu)$ and let $F\in B(\varpi_{\ell})$ such that $F\otimes T^0_\mu$ is semistandard. Then $$\Psi_{F\otimes T^0_\mu} = t^{\ell(u_F)}(T_{u_F^{-1}})^{-1}\one_{\varpi_\ell+\mu}\,.$$
	In particular, $$ \Psi_{T^0_\lambda} = \one_\lambda\,,\qquad \hbox{ for } \lambda \in (\ZZ_{\geq 0}^n)_+ \,.$$
\end{proposition}

\begin{proof}
	
	Suppose $C = (1,\ldots,\ell) \in B(\varpi_{\ell})$ is highest weight for some $\ell\in [n]$. Then $u_C = 1$, hence by \eqref{Psi_C}, $$ \Psi_{T^0_{\varpi_{\ell}}} = \Psi_C = t^{\ell(u_C)}(T_{u_C^{-1}})^{-1}\one_{\varpi_{\ell}} = \one_{\varpi_{\ell}} \,.  $$
	
	Let $k \geq 1$. Suppose $\lambda = \varpi_{\ell_{k+1}}+ \varpi_{\ell_k}+ \ldots + \varpi_{\ell_1}$, with $1 \leq \ell_1 \leq \ldots \leq \ell_{k} \leq \ell_{k+1} \leq n$. Let $\mu = \varpi_{\ell_k} + \ldots + \varpi_{\ell_1}$. Assume that $\Psi_{T^0_{\mu}} = \one_{\mu}$. 
	
	\noindent
	Since $S_{n,\varpi_{\ell_k}} = S_{[1,\ell_k]} \times S_{[\ell_k+1,n]}$, $S_{[1,\ell_1]}, \ldots, S_{[\ell_{k-1}+1,\ell_k]}, S_{[\ell_k+1,n]} \subseteq S_{n,\varpi_{\ell_k}}$. Then
	$$ \Psi_{T^0_\mu} = \one_\mu = \one_{\varpi_{\ell_1} + \ldots + \varpi_{\ell_{k}}} = \one_{[1,\ell_1]}\one_{[\ell_1+1,\ell_2]}  \ldots \one_{[\ell_{k-1}+1,\ell_k]}\one_{[\ell_k+1,n]} \in H_{n,\varpi_{\ell_{k}}} \,.$$
By \autoref{Snfactorization}, $$ \one_{[\ell_k+1,n]} = \Bigg(\sum_{\substack{F \in B(\varpi_{\ell_{k+1}}) \\ f_1 = 1,\ldots,f_{\ell_k}= \ell_k }}T_{u_F} \Bigg) \one_{[\ell_{k+1},n]} \,.$$  If $F \in B(\varpi_{\ell_{k+1}})$ is such that $f_1 =1,\ldots ,f_{\ell_k} = \ell_k$, then $wu_F = u_Fw$ for all $w  \in S_{[1,\ell_1]}\times \ldots \times S_{[\ell_{k-1}+1,\ell_k]}$. So, $T_w T_{u_F} = T_{u_F}T_w$ for all such $w$ and $F$.
Therefore, \begin{align*}
		\Psi_{T^0_{\mu}} &= \one_{[1,\ell_1]}\one_{[\ell_1+1,\ell_2]}  \ldots \one_{[\ell_{k-1}+1,\ell_k]}\one_{[\ell_k+1,n]} 
		\\
		&=  \one_{[1,\ell_1]}\one_{[\ell_1+1,\ell_2]}  \ldots \one_{[\ell_{k-1}+1,\ell_k]}\Bigg(\sum_{\substack{F \in B(\varpi_{\ell_{k+1}}) \\ f_1 = 1,\ldots,f_{\ell_k}= \ell_k }}T_{u_F} \Bigg) \one_{[\ell_{k+1},n]} 
		\\
		&= \Bigg(\sum_{\substack{F \in B(\varpi_{\ell_{k+1}}) \\ f_1 = 1,\ldots,f_{\ell_k}= \ell_k }}T_{u_F} \Bigg)\one_{[1,\ell_1]}\one_{[\ell_1+1,\ell_2]}  \ldots \one_{[\ell_{k-1}+1,\ell_k]}\one_{[\ell_{k+1},n]} = \sum_{\substack{F \in B(\varpi_{\ell_{k+1}}) \\ f_1 = 1,\ldots,f_{\ell_k}= \ell_k }}T_{u_F}  \one_{\lambda}\,.
	\end{align*}
	
	\noindent
	So by the definition of $\Psi$, $$\Psi_{F\otimes T^0_{\mu}} = t^{\ell(u_F)}(T_{u_F^{-1}})^{-1}\one_{\lambda}\,,$$  for  $F \in B(\varpi_{\ell_{k+1}})$ such that  $f_1=1,\ldots,f_{\ell_k}=\ell_k\,.$ 
	
	
\end{proof}

\subsection{The recursions for $\Psi_T$}

Let $C \in B(\varpi_{\ell})$ for some $\ell \in [n]$. Define $$\sigma^j_C = \begin{cases}
	-1, & \hbox{ if } s_jC > C,
	\\
	+1, & \hbox{ if } s_jC \leq C,
\end{cases} = \begin{cases}
	-1, & \hbox{ if } j \in C, j+1 \notin C\,,
	\\
	+1, & \hbox{ if } \hbox{ otherwise }.
\end{cases} $$


\begin{example}
	\ytableausetup{nosmalltableaux,centertableaux}
	$$ \hbox{ If }\, C = \ytableaushort{1,3,4,7}\,, \quad \hbox{ then } \qquad \sigma^1_C = -1,\, \sigma^2_C = +1,\, \sigma^3_C = +1,\, \sigma^4_C = -1,\, \sigma^5_C = +1\,.  $$
\end{example}

%
%

Let $1 \leq \ell_1 \leq \ldots \leq \ell_r \leq n$. For $T = C_r \otimes \ldots \otimes C_1 \in B(\varpi_{\ell_r})\otimes \ldots \otimes B(\varpi_{\ell_1})$, with $C_i \in B(\varpi_{\ell_i})$ for $i \in [r]$, and $k \in [r]$, define \begin{equation}
	\Omega^j_k(T) = C_r  \otimes \ldots \otimes C_{k+1} \otimes s_j C_k \otimes \ldots s_j C_1\,.
\end{equation} 

The following theorem is the Hecke algebra lift of Klostermann's recursions \autoref{prop:klostermannrecs}.

\begin{theorem}\label{Psirecrcols}
Let $T = C_r \otimes \ldots \otimes C_1 \in B(\varpi_{\ell_r})\otimes \ldots \otimes B(\varpi_{\ell_1})$, with $C_i \in B(\varpi_{\ell_i})$ for $i \in [r]$. 
	
	\begin{enumerate}
		\item[(a)] If $T$ is not semistandard then $\Psi_T = 0$.
		\item[(b)] If $T$ is highest weight of shape $\lambda (= \sum\limits_{i = 1}^{r} \varpi_{\ell_i})$, then $\Psi_T = \one_\lambda$.
		\item[(c)] Suppose $\mu \in (\ZZ_{\geq 0}^n)_+$ and $T \otimes T^0_\mu$ is semistandard. 
		
		If $j\notin C_1,j+1 \in C_1$ then 
		\begin{align}\label{eq:mainrec}
			\Psi_{T\otimes T^0_\mu} = \begin{cases}
				T_j\Psi_{\Omega^j_r(T)\otimes T^0_\mu} + \displaystyle \sum\limits_{\substack{r > k \geq 1 \\ \sigma^j_{C_k}\sigma^j_{C_{k+1}}=-1}} \sigma^j_{C_k} (1-t) \Psi_{\Omega^j_k(T)\otimes T^0_\mu}\,,& \hbox{ if } j \in C_r, j+1 \notin C_r \,,
				\vspace{1em}
				\\
				tT_j^{-1}\Psi_{\Omega^j_r(T)\otimes T^0_\mu} + \displaystyle \sum\limits_{\substack{r > k \geq 1 \\ \sigma^j_{C_k}\sigma^j_{C_{k+1}}=-1}} \sigma^j_{C_k} (1-t) \Psi_{\Omega^j_k(T)\otimes T^0_\mu}\,,&  \hbox{otherwise}\,.
			\end{cases}
		\end{align} 
		
	\end{enumerate}
\end{theorem}

We prove \autoref{Psirecrcols} by induction on the number of columns $r$. In the next subsections we work out the base cases of $r=1$ and $r=2$.

\subsection{One-Column}

\begin{proposition}\label{Psirecess1col}
\item[(1)] Let $\lambda \in (\ZZ_{\geq 0}^n)_+$ and $\ell \geq \ell(\lambda)$. Let $E \in B(\varpi_{\ell})$, with $E\otimes T^0_\lambda$ semistandard. If $j\notin E,j+1 \in E$, then $$ \Psi_{E\otimes T^0_\lambda} = tT_j^{-1} \Psi_{s_jE\otimes T^0_\lambda}\,. $$ 

\item[(2)] If $E \in B(\varpi_{\ell})$ and $s_jE \leq E$, then $$ \Psi_E = tT_j^{-1}\Psi_{s_jE} \,.$$
\end{proposition}

\begin{proof}

\item[(1)] 
	Suppose $j \notin E$, $j+1 \in E$. By \autoref{lemma:sjuE}, $u_E = s_ju_{s_jE}$ and $\ell(u_{E}) = \ell(u_{s_jE})+1$. Then \begin{align*}
		\Psi_{E\otimes T^0_\lambda} = t^{\ell(u_E)}(T_{u_E^{-1}})^{-1} \one_{\varpi_\ell+\lambda} = t^{\ell(u_{s_jE})+1}T_j^{-1}(T_{u_{s_jE}^{-1}})^{-1} \one_{\varpi_\ell+\lambda} = tT_j^{-1} \Psi_{s_jE\otimes T^0_\lambda} \,.
	\end{align*}

\item[(2)] If $s_jE < E$ then $j \notin E, j+1 \in E$, so this is the statement of (1) with $\lambda = 0$. 

If $s_jE = E$ then by \autoref{lemma:sjuE} $s_ju_E = u_Es_k$ where $k = u_E^{-1}(j)$. Since either both $j,j+1 \in E$ or $j,j+1 \notin E$, so $k\in [1,\ell-1]\cup [\ell+1,n-1]$ and $s_k \in S_{n,\varpi_\ell}$. Hence $tT_k^{-1}\one_{\varpi_{\ell}} = \one_{\varpi_{\ell}}$. Then \begin{align*}
			tT_j^{-1}\Psi_{s_jE} &= tT_j^{-1}\Psi_E =  tT_j^{-1} t^{\ell(u_E)}(T_{u_E^{-1}})^{-1}\one_{\varpi_{\ell}} = t^{\ell(u_E)+1}(T_{u_E^{-1}s_j})^{-1}\one_{\varpi_{\ell}}  \\&= t^{\ell(u_E)+1}(T_{s_ku_E^{-1}})^{-1}\one_{\varpi_{\ell}} 
			= t^{\ell(u_E)+1}(T_kT_{u_E^{-1}})^{-1}\one_{\varpi_{\ell}}  \\&= t^{\ell(u_E)}(T_{u_E^{-1}})^{-1}tT_k^{-1}\one_{\varpi_{\ell}} = t^{\ell(u_E)}(T_{u_E^{-1}})^{-1}\one_{\varpi_{\ell}} = \Psi_{E} \,.
		\end{align*}

\end{proof}

\subsection{Fundamental Lemma}

The following lemma is the source of the signs in \eqref{eq:mainrec}.

\begin{lemma}\label{TXPsi}
	Let $K \in B(\varpi_{\ell_r})\otimes \ldots \otimes B(\varpi_{\ell_1})$ with $1 \leq \ell_1 \leq 
	\ldots \leq \ell_r \leq n$. Let $n \geq \ell \geq \ell_r$. Then \begin{align*}
		& \qquad \qquad tT_j^{-1}\bigg(\sum_{C \in B(\varpi_\ell)} X^C \Psi_{C\otimes K}\bigg) = \sum_{D \in B(\varpi_\ell)} X^D f_{D,K},
		\\
		& \quad \hbox{ and } \quad  T_j\bigg(\sum_{C \in B(\varpi_\ell)} X^C \Psi_{C\otimes K}\bigg) = \sum_{D \in B(\varpi_\ell)} X^D g_{D,K}, 
	\end{align*}
	with $f_{D,K} \,,g_{D,K} \in H_n$ given by
	\begin{align*}
		f_{D,K} = \begin{cases}
			tT_j^{-1}\Psi_{s_jD\otimes K}\,, & \hbox{ if } s_jD \leq D\,,
			\\
			T_j \Psi_{s_jD\otimes K} + (1-t)\Psi_{D\otimes K}\,, & \hbox{ if } s_jD \geq D\,,
		\end{cases}
	\end{align*}
	and \begin{align*}
		g_{D,K} = \begin{cases}
			tT_j^{-1}\Psi_{s_jD\otimes K} - (1-t)\Psi_{D\otimes K}\,, & \hbox{ if } s_jD \leq D\,,
			\\
			T_j \Psi_{s_jD\otimes K}\,, & \hbox{ if } s_jD \geq D\,. 
		\end{cases}		
	\end{align*}
\end{lemma}

\begin{proof}
	
	Note that since $tT_j^{-1} = T_j + 1-t$, the cases are consistent when $s_jD = D$.
	
	Recall that for a column $C = (c_1,\ldots,c_\ell)$  the element $\vec{C} = \varepsilon_{c_1}+\ldots+ \varepsilon_{c_\ell} \in \ZZ^n$. 
	
	\noindent Then by \autoref{lemma:sjEcomparison}, \begin{align*}
		s_jC \leq C \quad \hbox{ if and only if } \quad (\vec{C},\alpha_i) \in \{0,-1\}\,,
		\\
		s_jC \geq C \quad \hbox{ if and only if } \quad  (\vec{C},\alpha_i) \in \{0,+1\}\,.
	\end{align*}
	
	
	\item[(a)] 
	Using \eqref{tT_iinv} and  \autoref{TiXmu},   
		
		\begin{align*}
			&\qquad \sum_{C} tT_j^{-1}X^C \Psi_{C\otimes K}
			\\ 
			&= \sum_{C: (\vec{C},\alpha_j) = 0} tT_j^{-1}X^C \Psi_{C\otimes K} + \sum_{C: (\vec{C},\alpha_j) = 1}(tT_j^{-1} X^C \Psi_{C\otimes K} + tT_j^{-1} X^{s_jC}\Psi_{s_jC\otimes K}  )
			\\
			&= \sum_{C: (\vec{C},\alpha_j) = 0}tT_j^{-1}X^C  \Psi_{C\otimes K} + \sum_{C: (\vec{C},\alpha_j) = 1}((T_j+1-t) X^C \Psi_{C\otimes K} + tT_j^{-1} X^{s_jC}\Psi_{s_jC\otimes K}  )
			\\
			&= \sum_{C: (\vec{C},\alpha_j) = 0}X^C tT_j^{-1} \Psi_{C\otimes K} + \sum_{C: (\vec{C},\alpha_j) = 1}((X^{s_jC}tT_j^{-1}+(1-t)X^C) \Psi_{C\otimes K} + X^C T_j \Psi_{s_jC\otimes K}  )
			\\
			&= \sum_{C: (\vec{C},\alpha_j) = 0}X^C tT_j^{-1} \Psi_{C\otimes K} 
			+ \sum_{C: (\vec{C},\alpha_j) = 1} X^C (T_j \Psi_{s_jC\otimes K} + (1-t)\Psi_{C\otimes K}) 
			\\
			&\qquad+ \sum_{C:(\vec{C},\alpha_j) = -1} X^C  tT_j^{-1} \Psi_{s_jC\otimes K}\,.
		\end{align*}
		\noindent
		Therefore, the coefficient of $X^{D}$ in $\sum\limits_{C} tT_j^{-1}X^C \Psi_{C\otimes K}$ is \begin{align*}
			f_{D,K} = \begin{cases}
				tT_j^{-1}\Psi_{s_jD\otimes K}\,, & \hbox{ if } s_jD \leq D\,,
				\\
				T_j \Psi_{s_jD\otimes K} + (1-t)\Psi_{D\otimes K}\,, & \hbox{ if } s_jD \geq D\,.
			\end{cases}
		\end{align*}
		
		\item[(b)] 
		Using \autoref{TiXmu} and \eqref{quadTi}, \begin{align*}
			&\sum\limits_C T_j X^C \Psi_{C\otimes K} = \sum\limits_{C: (\vec{C},\alpha_j) = 0} T_j X^C \Psi_{C\otimes K}+\sum\limits_{C: (\vec{C},\alpha_j)=1} (T_j X^C \Psi_{C\otimes K} + T_j X^{s_jC} \Psi_{s_jC\otimes K} )
			\\
			&= \sum\limits_{C: (\vec{C},\alpha_j) = 0} T_j X^C \Psi_{C\otimes K}+\sum\limits_{C: (\vec{C},\alpha_j)=1} (T_j X^C \Psi_{C\otimes K} + T_j t^{-1}T_j X^C T_j \Psi_{s_jC\otimes K} )
			\\
			&= \sum\limits_{C: (\vec{C},\alpha_j) = 0} T_j X^C \Psi_{C\otimes K}+\sum\limits_{C: (\vec{C},\alpha_j)=1} (T_j X^C \Psi_{C\otimes K} + t^{-1}((t-1)T_j+t) X^C T_j \Psi_{s_jC\otimes K} )
			\\
			&= \sum\limits_{C: (\vec{C},\alpha_j) = 0} X^C T_j  \Psi_{C\otimes K}+\sum\limits_{C: (\vec{C},\alpha_j)=1} (X^{s_jC}tT_j^{-1} \Psi_{C\otimes K} + ((t-1)X^{s_jC} + X^CT_j)  \Psi_{s_jC\otimes K} )
			\\
			&= \sum\limits_{C: (\vec{C},\alpha_j) = 0} X^C T_j  \Psi_{C\otimes K} + \sum\limits_{C: (\vec{C},\alpha_j)=1} X^C T_j \Psi_{s_jC\otimes K} 
			\\
			&\qquad+ \sum\limits_{C: (\vec{C},\alpha_j)=-1} X^C(tT_j^{-1}\Psi_{s_jC\otimes K} + (t-1)\Psi_{C\otimes K})\,.
		\end{align*}
		\noindent 
		 Therefore, the coefficient of $X^{D}$ in $\sum\limits_{C} T_j X^C \Psi_{C\otimes K}$ is \begin{align*}
			g_{D,K} = \begin{cases}
				tT_j^{-1}\Psi_{s_jD\otimes K} + (t-1)\Psi_{D\otimes K}\,, & \hbox{ if } s_jD \leq D\,,
				\\
				T_j \Psi_{s_jD\otimes K}\,, & \hbox{ if } s_jD \geq D\,.
			\end{cases}
		\end{align*}

\end{proof}	

\autoref{TXPsi} shows that multiplying by $tT_j^{-1}$ to $\sum\limits_{C \in B(\varpi_{\ell})}X^C \Psi_{C \otimes K} $ produces a `$+$' sign in every case, whereas, multiplying by $T_j$ produces a `$-$' sign when $s_jD<D$. 

%
%
%
%
%
%

\subsection{Two Columns}

\begin{proposition}\label{Psirecess2col}
	\item[(1)] Let $\mu \in (\ZZ_{\geq 0}^n)_+$. Let $n \geq b \geq a \geq \ell(\mu)$. Let $T^0_\mu$ be the highest weight tableau of shape $\mu$. Let $E \in B(\varpi_a), F \in B(\varpi_b)$ such that $E\otimes T^0_\mu$ is semistandard.   
	
	If $j\notin E,j+1 \in E$  then \begin{align}
		\Psi_{F\otimes E\otimes T^0_\mu} = \begin{cases}
			tT_j^{-1} \Psi_{s_jF\otimes s_jE\otimes T^0_\mu}\,, &\hbox{ if } s_jF \leq F\,,
			\\
			T_j \Psi_{s_jF\otimes s_jE\otimes T^0_\mu} + (1-t) \Psi_{F\otimes s_jE\otimes T^0_\mu}\,, & \hbox{ if } s_jF\geq F\,.
		\end{cases}
	\end{align}
	
	\item[(2)] Let $n \geq b \geq a \geq 1$. Let $E \in B(\varpi_a), F \in B(\varpi_b)$. If $s_jE \leq E$ then \begin{align}\label{eq:psirecmod2cols}
		\Psi_{F\otimes E} = \begin{cases}
			tT_j^{-1} \Psi_{s_jF\otimes s_jE}\,, &\hbox{ if } s_jF \leq F\,,
			\\
			T_j \Psi_{s_jF\otimes s_jE} + (1-t) \Psi_{F\otimes s_jE}\,, & \hbox{ if } s_jF\geq F\,.
		\end{cases}
	\end{align}
\end{proposition}

\begin{proof}
	Note that since $tT_j^{-1} = T_j + 1-t$, the cases are consistent when $s_jF = F$.
	
	\item[(1)] 
	
	Let $j \notin E, j+1 \in E$. Then by \autoref{Psirecess1col}[(1)] $$\Psi_{E\otimes T^0_\mu} = tT_j^{-1}\Psi_{s_jE\otimes T^0_\mu}\,.$$ Then using \eqref{PsiXomega}, \begin{align*}
		\sum_{F \in B(\varpi_{b})} X^F \Psi_{F\otimes E\otimes T^0_\mu} = \Psi_{E\otimes T^0_\mu} X^{\varpi_{b}} &= tT_j^{-1} \Psi_{s_jE\otimes T^0_\mu} X^{\varpi_{b}} =  tT_j^{-1} \sum_{F\in B(\varpi_{b})} X^F \Psi_{F\otimes s_jE\otimes T^0_\mu}\,.
	\end{align*}
	
	\noindent In the notation of \autoref{TXPsi}, taking $K = s_jE\otimes T^0_\mu$, we get $f_{F,K} = \Psi_{F\otimes E\otimes T^0_\mu}$. Then by \autoref{TXPsi}, if $s_jF \leq F$ then \begin{align*}
		\Psi_{F\otimes E\otimes T^0_\mu} = f_{F,K} = tT_j^{-1}\Psi_{s_jF\otimes K} = tT_j^{-1}\Psi_{s_jF\otimes s_jE\otimes T^0_\mu}\,,
	\end{align*} 
	and if $s_jF \geq F$ then \begin{align*}
		\Psi_{F\otimes E\otimes T^0_\mu} = f_{F,K} = T_j\Psi_{s_jF\otimes K} + (1-t)\Psi_{F\otimes K} = T_j\Psi_{s_jF\otimes s_jE\otimes T^0_\mu} + (1-t)\Psi_{F\otimes s_jE\otimes T^0_\mu}\,.
	\end{align*}
	
	
	\item[(2)] Let $s_jE \leq E$. Then by \autoref{Psirecess1col}[(2)] $$\Psi_{E} = tT_j^{-1}\Psi_{s_jE}\,.$$ Then the proof proceeds exactly as (1) with $\mu = 0$. 
	
%
\end{proof}

\subsection{proof of \autoref{Psirecrcols}}

%
%

	Part (a) and (b) are the statements of \autoref{10Xlambda}[(a)] and \autoref{Psiinitialcondition}.

	The proof of part (c) is by induction on the number of columns $r$.
	
\noindent	The base cases $r=1$ and $r=2$ are \autoref{Psirecess1col}[(1)] and \autoref{Psirecess2col}[(1)]. 
	
\noindent	Let $r \geq 2$. Assume that \eqref{eq:mainrec} holds for $r$. We will prove \eqref{eq:mainrec} for $r+1$ columns.
	
\noindent	Let $1 \leq \ell_1 \leq \ldots \leq \ell_r \leq \ell_{r+1} \leq n$. Let $T = C_{r+1} \otimes C_r \otimes \ldots \otimes C_1$ with $C_i \in B(\varpi_{\ell_i})$ for $i \in [r+1]$. Let $S = C_r \otimes \ldots \otimes C_1$. We will prove \eqref{eq:mainrec} for $T$ assuming the statement for $S$. 
	
\noindent	Let $\ell = \ell_{r+1}$, and $\sigma^j_k = \sigma^j_{C_k}$ for $k \in [r+1]$.
	
\noindent	By \eqref{PsiXomega}, $$ \Psi_{T \otimes T^0_\mu} = \Psi_{C_{r+1}\otimes S\otimes T^0_\mu} = \hbox{ coefficient of } X^{C_{r+1}} \hbox{ in } \Psi_{S\otimes T^0_\mu}X^{\varpi_{\ell}}\,.$$
	
	
	\textbf{Case 1: $\sigma^j_r = 1$.}

\noindent	By the induction assumption, $$ \Psi_{S\otimes T^0_\mu} = tT_j^{-1}\Psi_{\Omega^j_r(S)\otimes T^0_\mu} + \sum\limits_{\substack{r > k \geq 1 \\ \sigma^j_k\sigma^j_{k+1}=-1}} \sigma^j_k (1-t) \Psi_{\Omega^j_k(S)\otimes T^0_\mu} \,.$$ 
	\noindent
	Then, using \eqref{PsiXomega}, \begin{align}\label{eq:temp1}
		\Psi_{S\otimes T^0_\mu} X^{\varpi_{\ell}} = \sum_{C \in B(\varpi_{\ell})} tT_j^{-1}X^C \Psi_{C\otimes \Omega^j_r(S)\otimes T^0_\mu} + \sum_{C \in B(\varpi_{\ell})}\sum\limits_{\substack{r > k \geq 1 \\ \sigma^j_k\sigma^j_{k+1}=-1}} \sigma^j_k (1-t) X^C \Psi_{C\otimes \Omega^j_k(S)\otimes T^0_\mu}\,.
	\end{align}
	\noindent
	The first sum of \eqref{eq:temp1} is  \begin{align*}
		\sum_{C \in B(\varpi_{\ell})} tT_j^{-1}X^C \Psi_{C\otimes \Omega^j_r(S)\otimes T^0_\mu} = \sum_{D \in B(\varpi_{\ell})} X^D f_{D,\Omega^j_r(S)\otimes T^0_\mu}\,,
	\end{align*} where $f_{D,\Omega^j_r(S)\otimes T^0_\mu}$ is given by \autoref{TXPsi}.
	\noindent
	Then, using \eqref{PsiXomega}, \begin{align*}
		\Psi_{C_{r+1}\otimes S\otimes T^0_\mu} &= \hbox{ coefficient of } X^{C_{r+1}} \hbox{ in } \Psi_{S\otimes T^0_\mu}X^{\varpi_{\ell}} 
		\\
		&= f_{C_{r+1},\Omega^j_r(S)\otimes T^0_\mu} + \sum_{\substack{r > k \geq 1 \\ \sigma^j_k\sigma^j_{k+1}=-1}} \sigma^j_k (1-t) \Psi_{C_{r+1}\otimes \Omega^j_k(S)\otimes T^0_\mu} 
	\end{align*}
	
	If $\sigma^j_{r+1} = 1$, then by \autoref{TXPsi}, \begin{align*}
		\Psi_{C_{r+1}\otimes S\otimes T^0_\mu} = tT_j^{-1}\Psi_{s_jC_{r+1}\otimes \Omega^j_r(S)\otimes T^0_\mu} + \sum\limits_{\substack{r > k \geq 1 \\ \sigma^j_k\sigma^j_{k+1}=-1}} \sigma^j_k (1-t) \Psi_{C_{r+1}\otimes \Omega^j_k(S)\otimes T^0_\mu}\,.
	\end{align*}
	\noindent
	Since $\sigma^j_r = 1$, when $\sigma^j_{r+1} = 1$, there are no extra sign changes, hence the above equation agrees with \eqref{eq:mainrec}. 
	
	If $\sigma^j_{r+1} = -1$, then by \autoref{TXPsi}, \begin{align*}
		&\Psi_{C_{r+1}\otimes S\otimes T^0_\mu} 
		\\
		&= T_j \Psi_{s_jC_{r+1}\otimes \Omega^j_r(S)\otimes T^0_\mu} + (1-t)\Psi_{C_{r+1}\otimes \Omega^j_r(S)\otimes T^0_\mu} + \sum\limits_{\substack{r > k \geq 1 \\ \sigma^j_k\sigma^j_{k+1}=-1}} \sigma^j_k (1-t) \Psi_{C_{r+1}\otimes \Omega^j_k(S)\otimes T^0_\mu}\,.
	\end{align*}
	
	\noindent
	Since $\sigma^j_r = 1$, when $\sigma^j_{r+1} = -1$ then there is an extra sign change in $r$-th place. Hence, the above equation agrees with \eqref{eq:mainrec}. 
	
	
	\textbf{Case 2: $\sigma^j_{r} = -1$}.
	
\noindent	By the induction assumption, \begin{align*}
		\Psi_{S\otimes T^0_\mu} = T_j\Psi_{\Omega^j_r(S)\otimes T^0_\mu} + \sum\limits_{\substack{r > k \geq 1 \\ \sigma^j_k\sigma^j_{k+1}=-1}} \sigma^j_k (1-t) \Psi_{\Omega^j_k(S)\otimes T^0_\mu}
	\end{align*}
	\noindent
	Then, using \eqref{PsiXomega},  \begin{equation}\label{eq:temp2}
		\Psi_{S\otimes T^0_\mu} X^{\varpi_{\ell}} = 	\sum_{C \in B(\varpi_{\ell})}T_jX^C\Psi_{C\otimes \Omega^j_r(S)\otimes T^0_\mu} + \sum_{C \in B(\varpi_{\ell})} \sum\limits_{\substack{r > k \geq 1 \\ \sigma^j_k\sigma^j_{k+1}=-1}} \sigma^j_k (1-t) X^C \Psi_{C\otimes \Omega^j_k(S)\otimes T^0_\mu}\,. 
	\end{equation}
	\noindent
	The first sum of \eqref{eq:temp2} is \begin{align*}
		\sum_{C \in B(\varpi_{\ell})}T_jX^C\Psi_{C\otimes \Omega^j_r(S)\otimes T^0_\mu} = \sum_{D \in B(\varpi_{\ell})} X^D g_{D,\Omega^j_r(S)\otimes T^0_\mu}\,,
	\end{align*} where $g_{D,\Omega^j_r(S)\otimes T^0_\mu}$ is given by \autoref{TXPsi}.
	\noindent
	Then, using \eqref{PsiXomega}, \begin{align*}
		\Psi_{C_{r+1}\otimes S\otimes T^0_\mu} &= \hbox{ coefficient of } X^{C_{r+1}} \hbox{ in } \Psi_{S\otimes T^0_\mu} X^{\varpi_{\ell}}
		\\
		&= g_{C_{r+1},\Omega^j_r(S)\otimes T^0_\mu} + \sum_{\substack{r > k \geq 1 \\ \sigma^j_k\sigma^j_{k+1}=-1}} \sigma^j_k (1-t) \Psi_{C_{r+1}\otimes \Omega^j_k(S)\otimes T^0_\mu} 
\end{align*}	
	
	If $\sigma^j_{r+1} = 1$, then by \autoref{TXPsi}, 
	\begin{align*}
		&\Psi_{C_{r+1}\otimes S\otimes T^0_\mu} 
		\\
		&= tT_j^{-1}\Psi_{s_jC_{r+1}\otimes \Omega^j_r(S)\otimes T^0_\mu} -  (1-t)\Psi_{C_{r+1}\otimes \Omega^j_r(S)\otimes T^0_\mu} + \sum\limits_{\substack{r > k \geq 1 \\ \sigma^j_k\sigma^j_{k+1}=-1}} \sigma^j_k (1-t) \Psi_{C_{r+1}\otimes \Omega^j_k(S)\otimes T^0_\mu}\,.
	\end{align*}
	Since $\sigma^j_r = -1$, when $\sigma^j_{r+1} = 1$ then there is an extra sign change in $r$-th place. Hence, the above equation agrees with \eqref{eq:mainrec}. 
	
	If $\sigma^j_{r+1} = -1$, then by \autoref{TXPsi},
	\begin{align*}
		\Psi_{C_{r+1}\otimes S\otimes T^0_\mu} = T_j \Psi_{s_jC_{r+1}\otimes \Omega^j_r(S)\otimes T^0_\mu} + \sum\limits_{\substack{r > k \geq 1 \\ \sigma^j_k\sigma^j_{k+1}=-1}} \sigma^j_k (1-t) \Psi_{C_{r+1}\otimes \Omega^j_k(S)\otimes T^0_\mu}\,.
	\end{align*}
Since $\sigma^j_r = -1$, when $\sigma^j_{r+1} = -1$, there are no extra sign changes, hence, the above equation agrees with \eqref{eq:mainrec}.

\qed
\ytableausetup{smalltableaux}
\begin{example}
	Let $n = 3$ and $\lambda = (2,1,0)$. Then $\one_{(2,1,0)} = 1$. \begin{align*}
		&\Psi_{\ytableaushort{11,2}} = \one_{(2,1,0)} = 1\,,
		\\
		&\Psi_{\ytableaushort{11,3}} = tT_2^{-1}\Psi_{\ytableaushort{11,2}} = tT_2^{-1} \one_{(2,1,0)} = tT_2^{-1}\,,
		\\
		&\Psi_{\ytableaushort{12,2}} = tT_1^{-1}\Psi_{\ytableaushort{11,2}} = tT_1^{-1}\one_{(2,1,0)} = tT_1^{-1}\,,
		\\
		&\Psi_{\ytableaushort{12,3}} = T_1 \Psi_{\ytableaushort{21,3}} + (1-t)\Psi_{\ytableaushort{11,3}} = T_1 \cdot 0 + (1-t)\cdot tT_2^{-1} = (1-t)tT_2^{-1}\,,
		\\
		&\Psi_{\ytableaushort{22,3}} = tT_1^{-1} \Psi_{\ytableaushort{11,3}} = tT_1^{-1}tT_2^{-1},
		\\
		&\Psi_{\ytableaushort{13,2}} = T_2\Psi_{\ytableaushort{12,3}} + (1-t)\Psi_{\ytableaushort{12,2}} = T_2 \cdot (1-t)tT_2^{-1} + (1-t)\cdot tT_1^{-1} \\&\qquad\qquad= (1-t)(t+tT_1^{-1}) = (1-t)(1+T_1)\,,
		\\
		&\Psi_{\ytableaushort{13,3}} = tT_2^{-1}\Psi_{\ytableaushort{12,2}} = tT_2^{-1}tT_1^{-1},
		\\
		&\Psi_{\ytableaushort{23,3}} = tT_2^{-1}\Psi_{\ytableaushort{22,3}} = tT_2^{-1}tT_1^{-1}tT_2^{-1} 
	\end{align*} 
\end{example}

\section{Relation to Macdonald's formula}

In this section we derive that the projection of $\Psi_T$ onto $\ZZ[t]$ satisfies the same recursions as that of Macdonald's $\psi_T$, thus proving their equality.

Let $\lambda = \varpi_{\ell_r} + \ldots + \varpi_{\ell_1}$ with $1 \leq \ell_1\leq \ldots \leq \ell_r \leq n$. Let $T \in B(\varpi_{\ell_r}) \otimes \ldots \otimes B(\varpi_{\ell_1})$. Since $\Psi_T \in H_n$, and $T_w \one_0 = t^{\ell(w)}\one_0$ for $w \in S_n$, then there exists $\widetilde{\psi}_T \in \ZZ(t^{\pm 1})$ such that \begin{equation}
	\widetilde{\psi}_T\one_0 = \dfrac{1}{W_\lambda(t)} \Psi_T \one_0\,.
\end{equation}
Apriori, $\widetilde{\psi}_T$ is a rational function in $t^{\pm1}$. A consequence of \autoref{tildepsimainthm} is that $\widetilde{\psi}_T \in \ZZ[t]$.

Let $\lambda \in (\ZZ_{\geq 0}^n)_+$.  Then \eqref{Plambdadef} and \eqref{eq:10Xlambda} gives \begin{equation}
	 P_\lambda(t) = \sum_{T \in B(\lambda)} \widetilde{\psi}_{T}X^{T} \,.
\end{equation} 

Let $T = C_r \otimes \ldots \otimes C_1 \in B(\varpi_{\ell_r})\otimes \ldots \otimes B(\varpi_{\ell_1})$, with $C_i \in B(\varpi_{\ell_i})$ for $i \in [r]$. Suppose $T \otimes T^0_\mu$ is semistandard. 
Suppose that $j\notin C_1,j+1 \in C_1$. Assume that the shape of $T\otimes T^0_\mu$ is $\lambda$. Multiplying \eqref{eq:mainrec} with $\dfrac{1}{W_\lambda(t)}\one_0$ on the right and using $tT_j^{-1}\one_0 = \one_0$ and $T_j\one_0 = t\one_0$ we get 		 
\begin{align}\label{eq:tildepsirec}
	\widetilde{\psi}_{T\otimes T^0_\mu} = \begin{cases}
		t\widetilde{\psi}_{\Omega^j_r(T)\otimes T^0_\mu} + \sum\limits_{\substack{r > k \geq 1 \\ \sigma^j_{C_k}\sigma^j_{C_{k+1}}=-1}} \sigma^j_{C_k} (1-t) \widetilde{\psi}_{\Omega^j_k(T)\otimes T^0_\mu}\,,& \hbox{ if } j \in C_r, j+1 \notin C_r \,,
		\\
		\widetilde{\psi}_{\Omega^j_r(T)\otimes T^0_\mu} + \sum\limits_{\substack{r > k \geq 1 \\ \sigma^j_{C_k}\sigma^j_{C_{k+1}}=-1}} \sigma^j_{C_k} (1-t) \widetilde{\psi}_{\Omega^j_k(T)\otimes T^0_\mu}\,,&  \hbox{otherwise}\,.
	\end{cases}
\end{align}

\begin{theorem}\label{tildepsimainthm}
	\begin{enumerate}
		\item Suppose $1 \leq \ell_1 \leq \ldots \leq \ell_r \leq n$. If $T = C_r \otimes \ldots \otimes C_1$ with $C_i \in B(\varpi_{\ell_i})$ then \begin{equation}\label{eq:tildepsiTproduct}
			\widetilde{\psi}_T = \prod_{j = 1}^{r-1} \widetilde{\psi}_{C_{j+1}\otimes C_{j}} \,.
		\end{equation}
		
		\item Let $1 \leq a \leq b \leq n$. For $F \in B(\varpi_b), E \in B(\varpi_a)$, the $\widetilde{\psi}_{F\otimes E}$ are determined by the following recursions. \begin{enumerate}
			\item[(a)] If $F\otimes E$ is not semistandard then $\widetilde{\psi}_{F\otimes E} = 0$,
			\item[(b)] If $E = (1,\ldots,a)$ and $F\otimes E$ is semistandard then $\widetilde{\psi}_{F\otimes E} = 1$.
			
			\item[(c)] Suppose $F\otimes E$ is semistandard. If $s_jE \leq E$ then \begin{align}\label{eq:tildepsirec2cols}
				\widetilde{\psi}_{F\otimes E} = \begin{cases}
					t \widetilde{\psi}_{s_jF\otimes s_jE} + (1-t) \widetilde{\psi}_{F\otimes s_jE}\,, & \hbox{ if } s_jF \geq F\,,
					\\
					\widetilde{\psi}_{s_jF\otimes s_jE}\,, &\hbox{ if } s_jF \leq F \,.
				\end{cases}
			\end{align}
		\end{enumerate}
		
	\end{enumerate}	
\end{theorem}

\begin{proof}
	
	\item[(2.a)] \autoref{10Xlambda} gives that if $T$ is not semistandard then $\widetilde{\psi}_T = 0$. 
	
	\item[(2.b)] Follows from \autoref{Psiinitialcondition}, by using \eqref{eq:1lambda10}.
	
	\item[(2.c)] Let $\lambda = \varpi_b + \varpi_a$. Multiplying \eqref{eq:psirecmod2cols} with $\dfrac{1}{W_\lambda(t)}\one_0$ gives \eqref{eq:tildepsirec2cols}.

	\item[(1)] 
	 If $T$ is not semistandard then one of $C_{j+1} \otimes C_j$ is not semistandard. Hence both sides of \eqref{eq:tildepsiTproduct} is $0$.
		
		From \autoref{Psiinitialcondition}, using \eqref{eq:1lambda10}, we get \begin{equation}
			\widetilde{\psi}_{T^0_\lambda} = 1\,, \qquad \hbox{ for }  \lambda \in (\ZZ_{\geq 0}^n)_+ \,.
		\end{equation}
		
		Without loss of generality we can assume that $T$ is semistandard and $T = S \otimes T^0_\mu$ where $S = C_r \otimes \ldots \otimes C_1$ with $C_i \in B(\varpi_{\ell_i})$. We want to prove that \begin{equation}\label{eq:tildepsimultiplies}
			\widetilde{\psi}_{T} = \prod_{i = 1}^{r-1} \widetilde{\psi}_{C_{i+1}\otimes C_{i}} \,.
		\end{equation}   
		
		If $r=0$ or $r=1$ then \eqref{eq:tildepsimultiplies} is true since both sides equal 1. 
		
		Assume $r \geq 2$. Let $S' = C_{r-1} \otimes \ldots \otimes C_1$. Then to prove \eqref{eq:tildepsimultiplies} it is enough to show that \begin{equation*}\label{eq:tildepsitemp}
			\widetilde{\psi}_T = \widetilde{\psi}_{C_r \otimes C_{r-1}} \cdot \widetilde{\psi}_{S'\otimes T^0_\mu}.
		\end{equation*}
		
		Assume that \eqref{eq:tildepsimultiplies} is true whenever $1\leq r<m$. We will prove \eqref{eq:tildepsimultiplies} for $r = m$. 
		
		\textbf{Case I:} Suppose $C_1$ is highest weight. Then  $\widetilde{\psi}_{C_2\otimes C_1} = 1$, hence \eqref{eq:tildepsimultiplies} is true in this case by induction assumption. 
		
		\textbf{Case II:} Suppose $C_1$ is not highest weight. Suppose that \eqref{eq:tildepsimultiplies} for $r=m$ is true when $C_1$ is replaced by a column less than $C_1$ in the $\leq$ order on columns. Since $C_1$ is not highest weight, there exists $j\in [n-1]$ such that $j \notin C_1, j+1 \in C_1$. By induction assumption, \begingroup
		\begin{equation*}
			\widetilde{\psi}_{L_m\otimes\ldots\otimes L_2\otimes s_jC_1\otimes T^0_\mu} = \widetilde{\psi}_{L_m\otimes L_{m-1}}\ldots \widetilde{\psi}_{L_2\otimes s_jC_1} \,,
		\end{equation*}
		\endgroup for all columns $L_i \in B(\varpi_{\ell_i})$ for $i \in [2,m]$.
		
		Let $\sigma^j_i = \sigma^j_{C_i}$ for $i \in [m]$.  
				
		\underline{\textbf{Case 1:} $\sigma^j_{m} =1, \sigma^j_{m-1} = 1$}. 

		By \eqref{eq:tildepsirec2cols}, $$\widetilde{\psi}_{C_m\otimes C_{m-1}} = \widetilde{\psi}_{s_jC_m\otimes s_jC_{m-1}}\,.$$ 
		\noindent
		Then by \eqref{eq:tildepsirec} and the induction assumption \begin{align*}
			\widetilde{\psi}_{T} &= \widetilde{\psi}_{\Omega^j_m(S)\otimes T^0_\mu} + \sum_{\substack{1 \leq k<m-1\\ \sigma^j_k \sigma^j_{k+1}=-1}} \sigma^j_k (1-t)\widetilde{\psi}_{\Omega^j_k(S)\otimes T^0_\mu}
			\\
			&= \widetilde{\psi}_{s_jC_m\otimes s_jC_{m-1}}\widetilde{\psi}_{\Omega^j_{m-1}(S')\otimes T^0_\mu} + \widetilde{\psi}_{C_m\otimes C_{m-1}}\sum_{\substack{1 \leq k<m-1\\ \sigma^j_k \sigma^j_{k+1}=-1}}\sigma^j_k (1-t)\widetilde{\psi}_{\Omega^j_k(S')\otimes T^0_\mu}
			\\
			&= \widetilde{\psi}_{C_m\otimes C_{m-1}}\widetilde{\psi}_{\Omega^j_{m-1}(S')\otimes T^0_\mu} + \widetilde{\psi}_{C_m\otimes C_{m-1}}\sum_{\substack{1 \leq k<m-1\\ \sigma^j_k \sigma^j_{k+1}=-1}}\sigma^j_k (1-t)\widetilde{\psi}_{\Omega^j_k(S')\otimes T^0_\mu}
			\\
			&= \widetilde{\psi}_{C_m\otimes C_{m-1}}\Big(\widetilde{\psi}_{\Omega^j_{m-1}(S')\otimes T^0_\mu} + \sum_{\substack{1 \leq k<m-1\\ \sigma^j_k \sigma^j_{k+1}=-1}}\sigma^j_k (1-t)\widetilde{\psi}_{\Omega^j_{k}(S')\otimes T^0_\mu}\Big)
			\\
			&= \widetilde{\psi}_{C_m\otimes C_{m-1}} \widetilde{\psi}_{S' \otimes T^0_\mu}\,.
		\end{align*}
		
		\underline{\textbf{Case 2:} $\sigma^j_m = 1, \sigma^j_{m-1} = -1$}. 

		Since $s_j(s_jC_{m-1}) = C_{m-1} < s_jC_{m-1}$, and $s_j(s_jC_m) = C_m \geq s_jC_m$, by \eqref{eq:tildepsirec2cols}, $$\widetilde{\psi}_{s_jC_m\otimes s_jC_{m-1}} = t\widetilde{\psi}_{C_m\otimes C_{m-1}} + (1-t)\widetilde{\psi}_{s_jC_m\otimes C_{m-1}}\,.$$ 
		\noindent
		Then by \eqref{eq:tildepsirec} and the induction assumption,
		\begin{align*}
			\widetilde{\psi}_{T} 
			&= \widetilde{\psi}_{\Omega^j_m(S)\otimes T^0_\mu} - (1-t) \widetilde{\psi}_{\Omega^j_{m-1}(S)\otimes T^0_\mu} + \sum_{\substack{1 \leq k<m-1 \\ \sigma^j_k \sigma^j_{k+1}=-1}} \sigma^j_k (1-t) \widetilde{\psi}_{\Omega^j_k(S)\otimes T^0_\mu}
			\\
			&= \Big(\widetilde{\psi}_{s_jC_m\otimes s_jC_{m-1}}-(1-t)\widetilde{\psi}_{C_m\otimes s_jC_{m-1}}\Big)\widetilde{\psi}_{\Omega^j_{m-1}(S')\otimes T^0_\mu} 
			\\
			&\qquad+ \widetilde{\psi}_{C_m\otimes C_{m-1}} \sum_{\substack{1 \leq k<m-1 \\ \sigma^j_k \sigma^j_{k+1}=-1}} \sigma^j_k (1-t) \widetilde{\psi}_{\Omega^j_k(S')\otimes T^0_\mu}
			\\
			&= \widetilde{\psi}_{C_m\otimes C_{m-1}} \Big(t\widetilde{\psi}_{\Omega^j_{m-1}(S')\otimes T^0_\mu} + \sum_{\substack{1 \leq k<m-1 \\ \sigma^j_k \sigma^j_{k+1}=-1}} \sigma^j_k (1-t) \widetilde{\psi}_{\Omega^j_k(S')\otimes T^0_\mu}\Big)
			\\
			&= \widetilde{\psi}_{C_m\otimes C_{m-1}} \widetilde{\psi}_{S' \otimes T^0_\mu} \,.
		\end{align*}
		
		\underline{\textbf{Case 3:} $\sigma^j_m = -1, \sigma^j_{m-1} = 1$}. 
		
		By \eqref{eq:tildepsirec2cols}, $$\widetilde{\psi}_{C_{m}\otimes C_{m-1}} = t \widetilde{\psi}_{s_jC_m\otimes s_jC_{m-1}}+ (1-t)\widetilde{\psi}_{C_m\otimes s_jC_{m-1}}\,.$$
		\noindent
		Then by \eqref{eq:tildepsirec} and the induction assumption \interdisplaylinepenalty=10000 \begin{align*}
			\widetilde{\psi}_{T} 
			&= t \widetilde{\psi}_{\Omega^j_m(S)\otimes T^0_\mu} + (1-t) \widetilde{\psi}_{\Omega^j_{m-1}(S)\otimes T^0_\mu} +  \sum_{\substack{1 \leq k<m-1\\ \sigma^j_k\sigma^j_{k+1}=-1}}  \sigma^j_k (1-t) \widetilde{\psi}_{\Omega^j_k(S)\otimes T^0_\mu} 
			\\
			&= \Big(t \widetilde{\psi}_{s_jC_m\otimes s_jC_{m-1}}+(1-t)\widetilde{\psi}_{C_m\otimes s_jC_{m-1}}\Big)\widetilde{\psi}_{\Omega^j_{m-1}(S')\otimes T^0_\mu} 
			\\
			&\qquad + \widetilde{\psi}_{C_m\otimes C_{m-1}} \sum_{\substack{1 \leq k<m-1\\ \sigma^j_k\sigma^j_{k+1}=-1}}  \sigma^j_k (1-t) \widetilde{\psi}_{ \Omega^j_k(S')\otimes T^0_\mu}
			\\
			&= \widetilde{\psi}_{C_m\otimes C_{m-1}}\Big(\widetilde{\psi}_{\Omega^j_{m-1}(S')\otimes T^0_\mu} + \sum_{\substack{1 \leq k<m-1\\ \sigma^j_k\sigma^j_{k+1}=-1}}  \sigma^j_k (1-t) \widetilde{\psi}_{\Omega^j_k(S')\otimes T^0_\mu}
			\Big)
			\\
			&= \widetilde{\psi}_{C_m\otimes C_{m-1}} \widetilde{\psi}_{S'\otimes T^0_\mu}\,.
		\end{align*}
		\interdisplaylinepenalty=1000
		
		\underline{\textbf{Case 4:} $\sigma^j_m = -1, \sigma^j_{m-1} = -1$.} 
		
		Since $s_j(s_jC_{m-1}) = C_{m-1} < s_jC_{m-1}$, and $s_j(s_jC_m) = C_m < s_jC_m$, by \eqref{eq:tildepsirec2cols}, $$\widetilde{\psi}_{s_jC_m \otimes s_jC_{m-1}} = \widetilde{\psi}_{C_m\otimes C_{m-1}}\,.$$ 
		\noindent
		Then by \eqref{eq:tildepsirec} and the induction assumption
		\begin{align*}
			\widetilde{\psi}_{T} &=t \widetilde{\psi}_{\Omega^j_m(S)\otimes T^0_\mu} + \sum_{\substack{1 \leq k<m-1\\ \sigma^j_k\sigma^j_{k+1}=-1}} \sigma^j_k (1-t) \widetilde{\psi}_{ \Omega^j_k(S)\otimes T^0_\mu}
			\\
			&= t \widetilde{\psi}_{s_jC_m\otimes s_jC_{m-1}} \widetilde{\psi}_{\Omega^j_{m-1}(S')\otimes T^0_\mu} + \widetilde{\psi}_{C_m\otimes C_{m-1}}\sum_{\substack{1 \leq k<m-1\\ \sigma^j_k\sigma^j_{k+1}=-1}} \sigma^j_k (1-t) \widetilde{\psi}_{ \Omega^j_k(S')\otimes T^0_\mu}
			\\
			&= t \widetilde{\psi}_{C_m\otimes C_{m-1}} \widetilde{\psi}_{\Omega^j_{m-1}(S')\otimes T^0_\mu} + \widetilde{\psi}_{C_m\otimes C_{m-1}}\sum_{\substack{1 \leq k<m-1\\ \sigma^j_k\sigma^j_{k+1}=-1}} \sigma^j_k (1-t) \widetilde{\psi}_{ \Omega^j_k(S')\otimes T^0_\mu}
			\\
			&= \widetilde{\psi}_{C_m\otimes C_{m-1}} \Big(t\widetilde{\psi}_{\Omega^j_{m-1}(S')\otimes T^0_\mu} + \sum_{\substack{1 \leq k<m-1\\ \sigma^j_k\sigma^j_{k+1}=-1}}\sigma^j_k (1-t) \widetilde{\psi}_{ \Omega^j_k(S')\otimes T^0_\mu}\Big)
			\\
			&= \widetilde{\psi}_{C_m\otimes C_{m-1}} \widetilde{\psi}_{S'\otimes T^0_\mu}\,.
		\end{align*}

\end{proof}

\begin{remark}
	\eqref{eq:tildepsiTproduct} does not hold true if one replaces $\psi$ with $\Psi$. For example, if $n = 3$ then $\Psi_{T^0_{(3,2,0)}} = \one_{(3,2,0)} = 1$, and $\Psi_{T^0_{(2,2,0)}} = 1+T_1$, and $\Psi_{T^0_{(2,1,0)}} = 1$. So $$ \Psi_{\ytableaushort{111,22}} \neq \Psi_{\ytableaushort{11,22}}\cdot \Psi_{\ytableaushort{11,2}}\,. $$
\end{remark}

Comparing \autoref{prop:klostermannrecs} and \autoref{tildepsimainthm} we get
\begin{corollary}\label{tildepsi=psi}
	Let $T \in B(\varpi_{\ell_r})\otimes \ldots \otimes B(\varpi_{\ell_1})$ with $1 \leq \ell_1 \leq 
\ldots \leq \ell_r \leq n$. Then $$ \widetilde{\psi}_T = \psi_T\,.$$
\end{corollary}

\section{Examples}

%


\subsection{$\Psi_T$ for $T \in B((3,2,0))$}
\ytableausetup{smalltableaux,centertableaux}
\[\]
Since $\one_{(3,2,0)} = 1$,
\begin{align*}
	\\
	&\Psi_{\ytableaushort{111,22}} = \one_{(3,2,0)} = 1\,,
	\\[1em]
	&\Psi_{\ytableaushort{111,23}} = T_2 \Psi_{\ytableaushort{111,32}} + (1-t) \Psi_{\ytableaushort{111,22}} = 0+(1-t)\one_{(3,2,0)} = (1-t)\,,
	\\[1em]
	&\Psi_{\ytableaushort{111,33}} = tT_2^{-1}\Psi_{\ytableaushort{111,22}} = tT_2^{-1}\one_{(3,2,0)} = tT_2^{-1}\,,
	\\[1em]
	&\Psi_{\ytableaushort{112,22}} = tT_1^{-1}\Psi_{\ytableaushort{111,22}} = tT_1^{-1}\one_{(3,2,0)} = tT_1^{-1}\,, 
	\\[1em]
	&\Psi_{\ytableaushort{112,23}} = tT_1^{-1}\Psi_{\ytableaushort{121,23}} - (1-t) \Psi_{\ytableaushort{121,23}} + (1-t) \Psi_{\ytableaushort{111,23}} \\&\qquad\qquad= 0 + 0 + (1-t)\cdot (1-t) \one_{(3,2,0)} = (1-t)^2\,,
	\\[1em]
	&\Psi_{\ytableaushort{112,33}} = T_1 \Psi_{\ytableaushort{221,33}} + (1-t)\Psi_{\ytableaushort{111,33}} = 0 + (1-t)tT_2^{-1}\one_{(3,2,0)} = (1-t)tT_2^{-1}\,,
	\\[1em]
	&\Psi_{\ytableaushort{113,22}} = T_2 \Psi_{\ytableaushort{112,33}} + (1-t)\Psi_{\ytableaushort{112,22}} \\&\qquad \qquad= T_2\cdot (1-t)tT_2^{-1} \one_{(3,2,0)} + (1-t)tT_1^{-1}\one_{(3,2,0)} \\&\qquad\qquad= t(1-t)\one_{(3,2,0)} + (1-t)tT_1^{-1}\one_{(3,2,0)} =  t(1-t) + (1-t)tT_1^{-1}\,,
	\\[1em]
	&\Psi_{\ytableaushort{113,23}} = T_2\Psi_{\ytableaushort{112,32}} + (1-t)\Psi_{\ytableaushort{112,22}} = 0 + (1-t)tT_1^{-1}\one_{(3,2,0)} = (1-t)tT_1^{-1}\,,
	\\[1em]
	&\Psi_{\ytableaushort{113,33}} = tT_2^{-1}\Psi_{\ytableaushort{112,22}} = tT_2^{-1}\cdot(1-t)tT_1^{-1}\one_{(3,2,0)} = (1-t)tT_2^{-1}tT_1^{-1}\,,
	\\[1em]
	&\Psi_{\ytableaushort{122,23}} = tT_1^{-1}\Psi_{\ytableaushort{111,23}} = tT_1^{-1}(1-t)tT_1^{-1}\one_{(3,2,0)} = (1-t)tT_1^{-1}tT_1^{-1}\,,
	\\[1em]
	&\Psi_{\ytableaushort{122,33}} = T_1\Psi_{\ytableaushort{211,33}} + (1-t)\Psi_{\ytableaushort{111,33}} = 0 + (1-t)tT_2^{-1}\one_{(3,2,0)} = (1-t)tT_2^{-1} \,,
	\\[1em]
	&\Psi_{\ytableaushort{123,23}} = T_2\Psi_{\ytableaushort{122,33}} + (1-t)\Psi_{\ytableaushort{122,23}} \\&\qquad\qquad= T_2\cdot (1-t)tT_2^{-1}\one_{(3,2,0)} + (1-t)\cdot tT_1^{-1}(1-t)tT_1^{-1}\one_{(3,2,0)} \\&\qquad\qquad= (1-t)t + (1-t)^2tT_1^{-1}tT_1^{-1}\,,
	\\[1em]
	&\Psi_{\ytableaushort{123,33}} = tT_2^{-1}\Psi_{\ytableaushort{122,23}} = tT_2^{-1}\cdot tT_1^{-1}(1-t)tT_1^{-1}\one_{(3,2,0)} = (1-t)tT_2^{-1} tT_1^{-1}tT_1^{-1}\,,
	\\[1em]
	&\Psi_{\ytableaushort{222,33}} = tT_1^{-1}\Psi_{\ytableaushort{111,33}} = tT_1^{-1}\cdot tT_2^{-1}\one_{(3,2,0)} = tT_1^{-1}tT_2^{-1}\,,
	\\[1em]
	&\Psi_{\ytableaushort{223,33}} = tT_2^{-1}\Psi_{\ytableaushort{222,33}} = tT_2^{-1}\cdot tT_1^{-1} tT_2^{-1}\one_{(3,2,0)} = tT_2^{-1} tT_1^{-1} tT_2^{-1}\,.
	\end{align*}
	
\subsection{$\psi_T$ for $T \in B((3,2,0))$}
\[\]
Using $\psi_T\one_0 = \dfrac{1}{W_\lambda(t)}\Psi_T \one_0$,  
\ytableausetup{smalltableaux,centertableaux}
\begin{align*}
	\\
	&\psi_{\ytableaushort{111,22}} = 1\,,
	&\psi_{\ytableaushort{111,23}} = (1-t)\,,
	\\[1em]
	&\psi_{\ytableaushort{111,33}} =  1\,,
	&\psi_{\ytableaushort{112,22}} = 1\,, 
	\\[1em]
	&\psi_{\ytableaushort{112,23}} = (1-t)^2\,,
	&\psi_{\ytableaushort{112,33}} = (1-t)\,,
	\\[1em]
	&\psi_{\ytableaushort{113,22}} = t(1-t) + (1-t) = (1-t^2)\,,
	&\psi_{\ytableaushort{113,23}} =  (1-t)\,,
	\\[1em]
	&\psi_{\ytableaushort{113,33}} =  (1-t)\,,
	&\psi_{\ytableaushort{122,23}} =  (1-t)\,,
	\\[1em]
	&\psi_{\ytableaushort{122,33}} =  (1-t)\,,
	&\psi_{\ytableaushort{123,23}} = t (1-t) + (1-t)^2 = (1-t)\,,
	\\[1em]
	&\psi_{\ytableaushort{123,33}} = (1-t)\,,
	&\psi_{\ytableaushort{222,33}} =  1\,,
	\\[1em]
	&\psi_{\ytableaushort{223,33}} = 1\,.
\end{align*}

%


\subsection{Example of \eqref{eq:tildepsiTproduct}}

Suppose $T = C_r \otimes \ldots \otimes C_1$. If $T$ is not semistandard then $\widetilde{\psi}_T = 0$ and atleast one of $C_{i+1}\otimes C_i$ is not semistandard, so $\widetilde{\psi}_{C_{i+1}\otimes C_i} = 0$. Hence $\widetilde{\psi}_T = \prod\limits_{i=1}^{r-1} \widetilde{\psi}_{C_{i+1} \otimes C_i}$. 

Since $\widetilde{\psi}_{\ytableaushort{111,22}} = 1 = \widetilde{\psi}_{\ytableaushort{11,22}} \widetilde{\psi}_{\ytableaushort{11,2}}$, using \eqref{eq:tildepsirec2cols},

\begin{align*}
	\widetilde{\psi}_{\ytableaushort{111,23}} &= t\widetilde{\psi}_{\ytableaushort{111,32}} + (1-t) \widetilde{\psi}_{\ytableaushort{111,22}} = t\widetilde{\psi}_{\ytableaushort{11,32}}\widetilde{\psi}_{\ytableaushort{11,2}} + (1-t) \widetilde{\psi}_{\ytableaushort{11,22}}\widetilde{\psi}_{\ytableaushort{11,2}}
	\\
	&=\Big( t\widetilde{\psi}_{\ytableaushort{11,32}} + (1-t) \widetilde{\psi}_{\ytableaushort{11,22}} \Big) \widetilde{\psi}_{\ytableaushort{11,2}} = \widetilde{\psi}_{\ytableaushort{11,23}} \widetilde{\psi}_{\ytableaushort{11,2}}\,.
\end{align*}

Similarly, using $\widetilde{\psi}_{\ytableaushort{12,23}} = \widetilde{\psi}_{\ytableaushort{11,23}}$\,, and $\widetilde{\psi}_{\ytableaushort{12,3}} = t \widetilde{\psi}_{\ytableaushort{21,3}} + (1-t)\widetilde{\psi}_{\ytableaushort{11,3}}$\,, \begin{align*}
	\widetilde{\psi}_{\ytableaushort{112,23}} &= \widetilde{\psi}_{\ytableaushort{121,23}} - (1-t) \widetilde{\psi}_{\ytableaushort{121,23}} + \widetilde{\psi}_{\ytableaushort{111,23}} 
	\\
	&= \widetilde{\psi}_{\ytableaushort{12,23}} \widetilde{\psi}_{\ytableaushort{21,3}} - (1-t)\widetilde{\psi}_{\ytableaushort{12,23}} \widetilde{\psi}_{\ytableaushort{21,3}} + (1-t) \widetilde{\psi}_{\ytableaushort{11,23}}\widetilde{\psi}_{\ytableaushort{11,3}}
	\\
	&= \widetilde{\psi}_{\ytableaushort{12,23}} \Big( t \widetilde{\psi}_{\ytableaushort{21,3}} + (1-t)\widetilde{\psi}_{\ytableaushort{11,3}}  \Big) = \widetilde{\psi}_{\ytableaushort{11,23}} \widetilde{\psi}_{\ytableaushort{12,3}}\,.
\end{align*}

\bibliographystyle{alpha}
\bibliography{Klostermann_arxiv}

\end{document}